\documentclass[twoside,english]{elsarticle}
\usepackage[T2A,T1]{fontenc}
\usepackage{textcomp}
\usepackage[utf8]{inputenc}
\usepackage{color}
\usepackage{amsmath}
\usepackage{amsthm}
\usepackage{amssymb}

\makeatletter

\DeclareRobustCommand{\cyrtext}{%
  \fontencoding{T2A}\selectfont\def\encodingdefault{T2A}}
\DeclareRobustCommand{\textcyrillic}[1]{\leavevmode{\cyrtext #1}}

\theoremstyle{plain}
\newtheorem{thm}{\protect\theoremname}
\theoremstyle{plain}
\newtheorem{lem}[thm]{\protect\lemmaname}
\ifx\proof\undefined
\newenvironment{proof}[1][\protect\proofname]{\par
	\normalfont\topsep6\p@\@plus6\p@\relax
	\trivlist
	\itemindent\parindent
	\item[\hskip\labelsep\scshape #1]\ignorespaces
}{%
	\endtrivlist\@endpefalse
}
\providecommand{\proofname}{Proof}
\fi
\theoremstyle{remark}
\newtheorem{rem}[thm]{\protect\remarkname}
\theoremstyle{plain}
\newtheorem{cor}[thm]{\protect\corollaryname}
\theoremstyle{definition}
\newtheorem{defn}[thm]{\protect\definitionname}

\journal{Example: Nuclear Physics B}

\makeatother

\usepackage{babel}
\providecommand{\corollaryname}{Corollary}
\providecommand{\definitionname}{Definition}
\providecommand{\lemmaname}{Lemma}
\providecommand{\remarkname}{Remark}
\providecommand{\theoremname}{Theorem}

\begin{document}
\begin{frontmatter}
\title{Flow Representation of the Navier-Stokes Equations in Weighted Sobolev
Spaces}
\author[sek1,sek2]{Sekson Sirisubtawee }
\author[chu1]{Naowarat Manitcharoen }
\author[chu1,chu2]{Chukiat Saksurakan\corref{cor1}}
\ead{chukiat.s@cit.kmutnb.ac.th}
\cortext[cor1]{Corresponding author}
\address[sek1]{Department of Mathematics, Faculty of Applied Science, King Mongkut’s
University of Technology North Bangkok, Bangkok 10800, Thailand }
\address[sek2]{Centre of Excellence in Mathematics, CHE, Si Ayutthaya Road, Bangkok
10400, Thailand }
\address[chu1]{Department of Social and Applied Science, College of Industrial Technology,
King Mongkut's University of Technology North Bangkok, Bangkok 10800,
Thailand }
\address[chu2]{Center of Sustainable Energy and Engineering Materials (SEEM), College
of Industrial Technology, King Mongkut's University of Technology
North Bangkok, Bangkok 10800, Thailand }
\begin{abstract}
Using Constantin-Iyer representation also known more generally as
Euler-Lagrangian approach, we prove the local existence of the Navier-Stokes
equations in weighted Sobolev spaces with external forcing on $\mathbf{R}^{d}$,
for any dimension $d$ and $p$ such that $p>d\geq2.$ 
\end{abstract}
\end{frontmatter}

\section{Introduction}

The Navier-Stokes equations are widely regarded as one of the most
important equations in fluid mechanics due to their broad science
and engineering applications. While there is a huge amount of work
dedicated to the analysis of velocity equations, the Lagrangian approach
is less investigated. In Lagrangian settings, flow equations (positions
of individual particles) are derived and the solutions naturally allow
tracking of particles. In \citep{A,EM,E}, the Euler coordinates were
used to study motion of incompressible fluid on compact manifolds.
Fluid flows were treated as intrinsically defined infinitely dimensional
systems. In particular, Ebin and Marsden have shown in \citep{EM}
the local well-posedness of the deterministic Euler equations by solving
ODEs in the space of Sobolev volume-preserving diffeomorphisms. In
\citep{MV-1}, the author followed ideas of \citep{A,EM,E}, to study
the Euler equations with a random forcing term $\dot{W_{t}}$. Flow
equations were derived and solved for $d\geq2$ as ODEs in weighted
H\textcyrillic{ӧ}lder spaces. Without utilizing geometric tools, the
author dealt with spaces of diffeomorphisms analytically. More recently,
the authors of \citep{MMS} provided a stochastic framework extending
the geometric approach of Ebin and Mardens. Their results can be applied
to prove the local well-posedness of the stochastic Euler equations
with a random forcing term $\dot{W_{t}}$ in the Sobolev spaces with
$p=2$ and $s>\frac{d}{2}+1.$ 

The equivalence between the deterministic Navier-Stokes equations
and corresponding flow formulation was shown analytically in \citep{CI}.
Later, a self-contained proof for the local existence in H\textcyrillic{ӧ}lder
spaces was provided in \citep{I}. The main idea was to perturb the
flow by a Wiener process. By averaging out random trajectories, the
velocity can be recovered. While the results of \citep{CI} have motivated
a plethora of research into the Euler-Lagrangian approach of the Navier-Stokes
and Euler equations, only a handful has investigated it in Sobolev
spaces. In \citep{PR}, the Lagrangian formulation was used to prove
the local existence for the deterministic Euler equations in standard
Sobolev spaces $H_{p}^{s}$ with $p=2,$ $d\geq2$ and $s>\frac{d}{2}+1$
on the torus domain $\mathbb{T}^{d}:=\mathbb{R}^{d}/2\pi\mathbb{Z}^{d}.$
Under the same setting, their results were extended by \citep{OL}
to cover the stochastic Euler equations with Stratonovich transport
noise. 

While the geometric approach is elegant, the analytical approach is
arguably more accessible, especially for practitioners. For this reason,
we follow the idea of \citep{MV-1,CI,I} to study the Lagrangian formulation
of the Navier-Stokes equations. Our approach mostly relies on fundamental
results in harmonic analysis and can be easily understood by non-technical
readers. 

Our novel contribution can be summarized as follows:
\begin{itemize}
\item We study Constantin-Iyer representation for the Navier-Stokes equations
with random forcing $G\left(t\right)dt$ in the full space domain
instead of the Euler equations with/without Stratonovich transport
noise on tori \citep{PR,OL,FL19} or the Navier-Stokes equations without
forcing term on periodic domains \citep{CI,I}. Note that in \citep{CI},
the authors were well aware that $G\left(t\right)dt$ can be handled,
however, only formal discussion was provided. 
\item To the best of our knowledge, this work is the first to investigate
$L_{p}-$theory of Constantin-Iyer representation with general $p>d\geq2$
for the Navier-Stokes equations. Our results are new even without
the forcing term. We elaborate this further in the next point. 
\item Based on our Constantin-Iyer representation, we provide a self-contained
proof for the local existence of the Navier-Stokes equations in weighted
Sobolev spaces that can cover $p>d\geq2$ instead of non-weighted
Sobolev spaces with $p=2$ in \citep{PR,OL} or H\textcyrillic{ӧ}lder
spaces in \citep{CI,I,FL19}. Our work entails some additional analytical
results needed to handle the challenge of general $p$ in weighted
spaces compared to the case of $p=2$ in non-weighted spaces or the
H\textcyrillic{ӧ}lder spaces. We impose an assumption $l>1+\frac{d}{p}$
for the existence of the flow equations which is similar to $s>1+\frac{d}{2}$
found in \citep{PR,OL}. 
\item We also emphasize that our proof of existence covers the case of the
Euler equations namely $\epsilon=0$ without any special treatment
such as passing to the limit $\epsilon\rightarrow0.$
\end{itemize}
Let $\left(\Omega,\mathcal{F},\mathbf{P}\right)$ be a complete probability
space with a filtration of $\sigma$\textminus algebras $\mathbb{F}=\left(\mathcal{F}_{t},t\geq0\right)$
satisfying the usual conditions and $B_{t},W_{t}$ be independent
standard $d-$dimensional Wiener processes on $\left(\Omega,\mathcal{F},\mathbf{P}\right)$.
Let $\mathcal{\mathbb{F}}^{W}=\left(\mathcal{F}_{t}^{W},t\geq0\right)$
be the standard sub-filtration of $\mathbb{F}$ generated by $W_{t}$.
We will prove the local existence of the velocity $\mathbf{u}:\Omega\times\left(0,\infty\right)\times\mathbf{R}^{d}\rightarrow\mathbf{R}^{d}$
which is $\mathcal{\mathbb{F}}^{W}-$adapted and evolves with $t$
in a weighted Sobolev space according to the following Navier-Stokes
equation with $\epsilon\geq0$ (see Theorem \ref{thm:main_velocity}),
\begin{eqnarray*}
d\mathbf{u}\left(t\right) & = & \left[\mathcal{S}\left(-u^{k}\left(t\right)\partial_{k}\mathbf{u}\left(t\right)\right)+\frac{\epsilon^{2}}{2}\Delta\mathbf{u}\left(t\right)+G\left(t\right)\right]dt\\
\mathbf{u}\left(0\right) & = & \mathbf{u_{0}},\hspace{1em}\text{div }\mathbf{u}=0,\hspace{1em}t>0.
\end{eqnarray*}

The symbol $\mathcal{S}$ stands for the solenoidal projection. We
assume that the initial datum $\mathbf{u}_{0}$ is $\mathcal{F}_{0}^{W}-$measurable,
$\text{div }\mathbf{u}_{0}=0$ and the external forcing term $G:\Omega\times\left(0,\infty\right)\times\mathbf{R}^{d}\rightarrow\mathbf{R}^{d}$
is $\mathcal{\mathbb{F}}^{W}-$adapted. 

Similar to \citep{I}, we define the perturbed flow $\eta:\Omega\times\left(0,\infty\right)\times\mathbf{R}^{d}\rightarrow\mathbf{R}^{d}$
by 
\[
d\mathbf{\eta}\left(t\right)=\mathbf{u}\left(t,\eta\left(t\right)\right)dt+\epsilon dB_{t},\,\eta\left(0\right)=e,\,t>0.
\]

\noindent Here, $e$ denotes the identity map on $\mathbf{R}^{d}.$ 

We derive (see Lemma \ref{lem:flow1} and Remark \ref{rem:formal_flow_derivation})
the following equation for the flow $\eta$, denoting $\mathbf{\kappa}\left(t\right)=\eta^{-1}\left(t\right)$
the spatial inverse of $\eta\left(t\right)$, 
\begin{align*}
d\eta\left(t\right) & =\left.\mathbf{E}\left[\mathcal{S}\left(\nabla\kappa\left(t,z\right)\right)^{\ast}\mathbf{g_{\eta}}\left(t,\kappa\left(t,z\right)\right)|\mathcal{F}_{t}^{W}\right]\right|_{z=\eta\left(t\right)}dt+\epsilon dB_{t}\\
\eta\left(0\right) & =e,\hspace{1em}t>0.
\end{align*}
where
\[
\mathbf{g}_{\eta}\left(t,x\right)=\mathbf{u}_{0}+\int_{0}^{t}\left(\nabla\eta\left(s,x\right)\right)^{\ast}G\left(s,\eta\left(s,x\right)\right)ds,\hspace{1em}\left(\omega,t,x\right)\in\Omega\times\left(0,\infty\right)\times\mathbf{R}^{d}.
\]
Once we prove the local existence of $\eta$, the velocity can be
recovered from the flow via the formula,
\[
\mathbf{u}\left(t\right)=\mathbf{E}\left[\mathcal{S}\left(\nabla\kappa\left(t\right)\right)^{\ast}\mathbf{g}_{\eta}\left(t,\kappa\left(t\right)\right)|\mathcal{F}_{t}^{W}\right].
\]

Lastly, we note that $G$ can be random, although this does not cause
any additional difficulties compared to the case of deterministic
$G$. The fact that $G$ can be $\mathcal{\mathbb{F}}^{W}-$adapted
heuristically allows passing to the limit to the stochastic equations
with a stochastic integral as a forcing term, i.e., $G\left(t\right)dW_{t}$
instead of $G\left(t\right)dt$.

\section{Notation}

We list some commonly used notations in this paper.
\begin{itemize}
\item $\mathbf{E}_{t}^{W}\left(X\right)$ denotes the conditional expectation
$\mathbf{E}\left(X|\mathcal{F}_{t}^{W}\right).$
\item We assume throughout this paper that $d\geq2.$ 
\item $\mathbf{N}_{0}^{d}$ denotes the set of all $d-$dimensional multi-indices.
\item $e$ denotes an identity map from $\mathbf{R}^{d}$ to $\mathbf{R}^{d}.$
$I$ denotes the $d\times d$ identity matrix.
\item For any matrix or vector $A$, $A^{\ast}$ denotes its transpose.
\item $C_{0}^{\infty}=C_{0}^{\infty}\left(\mathbf{R}^{d}\right)$ denotes
the set of all indefinitely differentiable real-valued functions on
$\mathbf{R}^{d}$ with compact support.
\item $\left|\cdot\right|$ denotes standard Euclidean norms for both vectors
and matrices, regardless of dimensions.
\item For a real-valued function $f$ on $\mathbf{R}^{d}$, $\left|f\right|_{\infty}=\text{sup}{}_{x\in\mathbf{R}^{d}}\left|f\left(x\right)\right|.$
It is generalized to vector-valued and matrix-valued functions by
taking the maximum of $\left|\cdot\right|_{\infty}$ of all entries. 
\item For a real-valued function $f$ on $\left[0,\infty\right)\times\mathbf{R}^{d}$,
its partial derivatives are denoted by $\partial_{t}f=\partial f/\partial t,\,$$\partial_{i}f=\partial f/\partial x_{i}$,
$\partial_{ij}^{2}f=\partial^{2}f/\partial x_{i}\partial x_{j},$
$Df=\nabla f=\left(\partial_{1}f,...,\partial_{d}f\right)$. Given
a multi-index $\gamma\in\mathbf{N}_{0}^{d}$, $D^{\gamma}f=\partial^{\gamma}f=\frac{\partial^{\left|\gamma\right|}f}{\partial x_{1}^{\gamma_{1}}...\partial x_{d}^{\gamma_{d}}}$
and the same notations is used for weak derivatives. 
\item For $\mathbf{f}=\left(f^{1},f^{2},...,f^{m}\right)^{\ast}:\mathbf{R}^{d}\rightarrow\mathbf{R}^{m}$
and a multi-index $\gamma\in\mathbf{N}_{0}^{d}$, $D^{\gamma}\mathbf{f}=\left(D^{\gamma}f^{1},...,D^{\gamma}f^{m}\right)^{\ast}$
denotes its partial derivative and $D\mathbf{f}=\nabla\mathbf{f}=\left(\partial_{j}f^{i}\right)_{i,j}$
denotes its Jacobian matrix. The notation $D^{\gamma}\mathbf{f}$
is also extended to a matrix-valued function by entry-wise differentiation.
We will also write $\left\Vert \mathbf{f}\right\Vert =\left\vert \mathbf{f}\left(0\right)\right\vert +\left\vert \nabla\mathbf{f}\right\vert _{\infty}.$ 
\item If $\mathbf{f}$ is a real-valued, vector-valued or matrix-valued
function on $\mathbf{R}^{d}$, we denote $D^{k}\mathbf{f}=\left(D^{\gamma}\mathbf{f}\right)_{\left|\gamma\right|=k},k=1,2,3,...$
the tensor of all derivatives of order $k.$ 
\item $C^{n}\left(\mathbf{R}^{d}\right)=C^{n}\left(\mathbf{R}^{d};\mathbf{R}\right),\,n\geq0$
denotes the set of all $n$-times continuously differentiable functions
on $\mathbf{R}^{d}$ endowed with the finite norm $\left|f\right|_{C^{n}}=\sum_{0\leq\left|\gamma\right|\leq n}\sup_{x}\left|D^{\gamma}f\left(x\right)\right|<\infty.$ 
\item $\mathcal{C}^{\alpha}\left(\mathbf{R}^{d}\right)=\mathcal{C}^{\alpha}\left(\mathbf{R}^{d};\mathbf{R}\right),\,\alpha>0$
denotes the standard H\textcyrillic{ӧ}lder spaces on $\mathbf{R}^{d}$
endowed with the finite norm
\[
\left|f\right|_{\mathcal{C}^{\alpha}}=\left|f\right|_{C^{n}}+\sup_{x\neq y}\frac{\left|f\left(x\right)-f\left(y\right)\right|}{\left|x-y\right|^{\beta}}
\]
where $\alpha=n+\beta,$ $n$ is an integer and $\beta\in\left(0,1\right].$
\item $C^{n}$ (resp. $\mathcal{C}^{n}$) is extended to the space vector-valued
and matrix-valued functions whose all components belong to $C^{n}$
(resp. $\mathcal{C}^{n}$.) It is endowed with the the maximum of
$C^{n}-$norm (resp. $\mathcal{C}^{n}-$norm) of all entries. The
same notation $\left|f\right|_{C^{n}}$ (resp. $\left|f\right|_{\mathcal{C}^{n}}$)
are used. For the spaces, we will write $C^{n}\left(\mathbf{R}^{d};B\right)$
and $\mathcal{C}^{n}\left(\mathbf{R}^{d};B\right)$ with appropriate
$B$, e.g., $B=\mathbf{R}^{d}.$
\item For a multi-linear continuous operator $F:E\rightarrow F$ where $E,F$
are Banach spaces, $\left\Vert F\right\Vert $ denotes its operator
norm.
\item $C,N,M\in\left(0,\infty\right)$ with or without subscriptions denote
constants which generally change from line to line or even within
the same line. 
\item We will use the Einstein summation convention over repeated indices
when there is no chance of confusion.
\item When two functions are equal almost everywhere, we generally refer
to them by the same notation. If some properties of a function hold
after a modification on a set of measure zero, then we simply say
that a function satisfies such properties. In particular, if a function
is H\textcyrillic{ӧ}lder continuous or continuously differentiable
after a modification on a set of measure zero due to Sobolev embedding
theorem, then we simply say that it has such properties. 
\end{itemize}

\section{Function Spaces and Decomposition of Vector Fields}

\subsection{Function Spaces}

We will use Sobolev function spaces with asymptotic conditions originally
introduced by Cantor in \citep{c1} and were used in \citep{c2} to
prove the local well-posedness of the deterministic Euler equations. 

For $p>1,$ $l\geq0$, $\delta\in\mathbf{R}$, we denote by $H_{\delta,p}^{l}\left(\mathbf{R}^{d}\right)=H_{\delta,p}^{l}\left(\mathbf{R}^{d};\mathbf{R}\right)$
the space of real-valued functions $f:\mathbf{R}^{d}\rightarrow\mathbf{R}$
whose weak derivatives\textcolor{red}{{} }have the finite norm 

\[
\left|f\right|_{H_{\delta,p}^{l}}=\sum_{k=0}^{l}\sum_{\left|\gamma\right|=k}\left(\int_{\mathbf{R}^{d}}w^{p\left(\delta-l+k-d/p\right)}\left(x\right)\left|D^{\gamma}f\left(x\right)\right|^{p}dx\right)^{1/p}
\]

\noindent where $w\left(x\right)=\left(1+\left|x\right|^{2}\right)^{1/2}$.
By interpolation inequalities, the norm $\left|f\right|_{H_{\delta,p}^{l}}$
is equivalent to the norm, 
\[
\left(\int_{\mathbf{R}^{d}}w^{p\left(\delta-l-d/p\right)}\left(x\right)\left|f\left(x\right)\right|^{p}dx\right)^{1/p}+\sum_{\left|\gamma\right|=l}\left(\int_{\mathbf{R}^{d}}w^{p\left(\delta-d/p\right)}\left(x\right)\left|D^{\gamma}f\left(x\right)\right|^{p}dx\right)^{1/p}.
\]

It is easy to show that $C_{0}^{\infty}\left(\mathbf{R}^{d}\right)$
is dense in $H_{p,\delta}^{l}\left(\mathbf{R}^{d}\right)$ (e.g.,
\citep[Proposition 2.3.1]{s}.) If $\mathbf{v}$ is a vector, a matrix
or even a multi-dimensional tensor, the norm $\left|\mathbf{v}\right|_{H_{\delta,p}^{l}}$
is similarly defined by intrepreting $\left|\cdot\right|$ as the
Euclidean norm. The corresponding spaces are denoted with respect
to the dimensions of the range spaces, for example, $H_{\delta,p}^{l}\left(\mathbf{R}^{d};\mathbf{R}^{m}\right)$
or $H_{\delta,p}^{l}\left(\mathbf{R}^{d};\mathbf{R}^{m}\times\mathbf{R}^{n}\right).$

For $0<T\leq\infty$, we denote by $\sum_{\delta,p}^{l}\left(T\right)$
the space of $H_{\delta,p}^{l}-$valued functions $\mathbf{v}$ on
$\left[0,T\right]$ with the finite norm $\sup_{t\in\left[0,T\right]}\left|\mathbf{v}\left(t\right)\right|_{H_{\delta,p}^{l}}.$
We use the same notation $\sum_{\delta,p}^{l}\left(T\right)$ regardless
of dimensions as there is no chance of confusion.

If $\delta=0$, we write $H_{p}^{l}\left(\mathbf{R}^{d};B\right)=H_{0,p}^{l}\left(\mathbf{R}^{d};B\right)$
and if $\delta=l=0$, we write $L_{p}\left(\mathbf{R}^{d};B\right)=H_{0,p}^{0}\left(\mathbf{R}^{d};B\right)$
with an appropriate $B.$ 

We use $\left|\cdot\right|_{p}$ to denote the $L_{p}$ norms regardless
of dimensions. The estimates of types $\left|fg\right|_{p}\leq C\left|f\right|_{\infty}\left|g\right|_{p}$
and $\left|fg\right|_{\infty}\leq C\left|f\right|_{\infty}\left|g\right|_{\infty}$
with some generic $C$ will be used abundantly in this paper. We will
not be pedantic about the value of $C$ even when $C=1$, since it
depends on exact dimensions of $f$ and $g.$

Whenever $p>d$, we will interpret all derivatives of $f\in H_{p}^{l}\left(\mathbf{R}^{d}\right)$
and thus of $f\in H_{\delta,p}^{l}\left(\mathbf{R}^{d}\right)$ when
$\delta\geq l+\frac{d}{p}$ as classical derivatives (cf. \citep[Theorem 5 of Chapter 5]{Evans}.)
Consequently, we may rely on results in fundamental calculus such
as chain rules, product rules, Taylor's expansion, and integration
by-parts. 

\subsection{Decomposition of Vector Fields}

\noindent Recall the definition of the Newtonian potential, 
\[
\Gamma\left(x\right)=\Gamma^{d}\left(\left\vert x\right\vert \right)=\begin{cases}
\frac{\left\vert x\right\vert ^{2-d}}{d\left(2-d\right)\omega_{d}}, & d>2\\
\frac{1}{2\pi}\ln\left\vert x\right\vert , & d=2
\end{cases}
\]

\noindent and 
\[
\Gamma_{i}\left(x\right)=\frac{\partial}{\partial x_{i}}\Gamma\left(x\right)=\frac{1}{d\omega_{d}}\frac{x_{i}}{\left\vert x\right\vert }\left\vert x\right\vert ^{1-d},\,x\neq0,\hspace{1em}d\geq2
\]

\noindent where $\omega_{d}$ is the volume of the unit ball in $\mathbf{R}^{d}$.

Define the operators
\[
T_{i}f\left(x\right)=\int_{\mathbf{R}^{d}}\Gamma_{i}\left(x-y\right)f\left(y\right)dy,\,f\in C_{0}^{\infty}\left(\mathbf{R}^{d}\right),\,i=1,\ldots,d.
\]
We will use $T_{i},i=1,...,d$ to define gradient and solenoidal projections
for $\mathbf{v}\in H_{\theta+l,p}^{l}\left(\mathbf{R}^{d};\mathbf{R}^{d}\right)$
respectively. First, we set for $\mathbf{v}=\left(v^{1},....,v^{d}\right)^{\ast}\in C_{0}^{\infty}\left(\mathbf{R}^{d};\mathbf{R}^{d}\right),$
\begin{eqnarray*}
\mathcal{G}\left(\mathbf{v}\right) & = & \nabla T_{i}\left(v^{i}\right)^{\ast},\\
\mathcal{S}\left(\mathbf{v}\right) & = & \mathbf{v}-\mathcal{G}\left(\mathbf{v}\right)
\end{eqnarray*}
where the standard summation convention over repeated indices is assumed.
It is well-known that
\[
\mathcal{G}\left(\mathbf{v}\right)=\nabla T_{i}\left(v^{i}\right)^{\ast}=-RR_{j}v^{j},
\]
where $R_{j}f=-i\mathcal{F}^{-1}\left(\frac{\xi_{j}}{\left\vert \xi\right\vert }\mathcal{F}f\right)$
is the Riesz transform of $f$ and $R=\left(R_{1},...,R_{d}\right)^{\ast}.$
Usually, $\mathcal{G}\left(\mathbf{v}\right)$ and $\mathcal{S}\left(\mathbf{v}\right)$
are referred to as gradient and solenoidal projections of the vector
field $\mathbf{v}$ respectively, and
\[
\int\mathcal{G}\left(\mathbf{v}\right)\cdot\mathcal{S}\left(\mathbf{f}\right)dx=0,\,\mathbf{v\text{,}\,\mathbf{f}\in}C_{0}^{\infty}\left(\mathbf{R}^{d};\mathbf{R}^{d}\right).
\]

In fact, $\mathcal{G}\left(\mathbf{v}\right)=\nabla T_{i}\left(v^{i}\right)^{\ast}$
and $\mathcal{S}\left(\mathbf{v}\right)=\mathbf{v}-\mathcal{G}\left(\mathbf{v}\right)$
are continuous in $L_{p}\left(\mathbf{R}^{d};\mathbf{R}^{d}\right)$
i.e. $\left|\mathcal{G}\left(\mathbf{v}\right)\right|_{p}\leq C\left|\mathbf{v}\right|_{p}$
and $\left|\mathcal{S}\left(\mathbf{v}\right)\right|_{p}\leq C\left|\mathbf{v}\right|_{p}$
(see \citep[Remark 3.5]{M2}.) For more detailed discussion on projections
in non-weighted Sobolev spaces see for instance \citep[Section 3.1.2]{rm},
\citep[Section 3.2]{M2} and references therein. The following Helmholtz
decomposition for weighted Sobolev spaces is a key result for our
main proof. 
\begin{lem}
\label{lem:mik1_smoothing_decomp}Let $p>1,\,l\geq0,\,\theta\in\left(1,d\right)$.
The operators 
\[
T_{i}:H_{\theta+l,p}^{l}\left(\mathbf{R}^{d}\right)\rightarrow H_{\theta+l,p}^{l+1}\left(\mathbf{R}^{d}\right),\hspace{1em}i=1,...,d
\]
are bounded. That is for any $i=1,\ldots,d,$ there exists $C>0$
such that 
\[
\left|T_{i}f\right|_{H_{\theta+l,p}^{l+1}}\leq C\left\vert f\right\vert _{H_{\theta+l,p}^{l}},\,f\in H_{\theta+l,p}^{l}\left(\mathbf{R}^{d}\right).
\]

\noindent Consequently, 
\[
\mathcal{G}:H_{\theta+l,p}^{l}\left(\mathbf{R}^{d};\mathbf{R}^{d}\right)\rightarrow H_{\theta+l,p}^{l}\left(\mathbf{R}^{d};\mathbf{R}^{d}\right),\,\mathcal{S}:H_{\theta+l,p}^{l}\left(\mathbf{R}^{d};\mathbf{R}^{d}\right)\rightarrow H_{\theta+l,p}^{l}\left(\mathbf{R}^{d};\mathbf{R}^{d}\right)
\]

\noindent are linear continuous. Moreover,
\[
H_{\theta+l,p}^{l}\left(\mathbf{R}^{d};\mathbf{R}^{d}\right)=\mathcal{G}\left(H_{\theta+l,p}^{l}\left(\mathbf{R}^{d};\mathbf{R}^{d}\right)\right)\oplus\mathcal{S}\left(H_{\theta+l,p}^{l}\left(\mathbf{R}^{d};\mathbf{R}^{d}\right)\right)
\]
 and $\mathcal{S}\left(H_{\theta+l,p}^{l}\left(\mathbf{R}^{d};\mathbf{R}^{d}\right)\right)=\left\{ \mathbf{v}\in H_{\theta+l,p}^{l}\left(\mathbf{R}^{d};\mathbf{R}^{d}\right)|\hspace{1em}\text{div }\mathbf{v}=0\right\} .$
\end{lem}
\begin{proof}
The first estimate is an immediate consequence of Lemma \ref{Lem:lle3}
whose proof is provided fully in the Appendix. Regarding the direct
sum and divergence-free vector fields, we refer to \citep[Lemma 3.7]{M2}
and its proof which clearly carry over to weighted Sobolev spaces.
\end{proof}
\begin{rem}
\label{rem:sol_0}The solonoidal projection has a convenient formula
in weighted Sobolev spaces. Indeed, if $f\in H_{\theta+l,p}^{l}\left(\mathbf{R}^{d};\mathbf{R}^{d}\right)$
for some $p>1,\,l\geq0,\,\theta\in\left(1,d\right)$ and $u^{ij}=\partial_{i}N\left(f^{j}\right)$,
then by Lemma \ref{Lem:lle3}, $u\in H_{\theta+l,p}^{l+1}\left(\mathbf{R}^{d};\mathbf{R}^{d}\times\mathbf{R}^{d}\right).$
Following \citep[ Proposition 2]{MV-1}, we let
\[
\mathcal{G}f=\sum_{i}\nabla T_{i}\left(f^{i}\right)^{\ast}=\sum_{i}\nabla\text{\ensuremath{\left(u^{ii}\right)}}^{\ast},\,\left(\tilde{\mathcal{S}}f\right)^{j}=\sum_{i}\frac{\partial}{\partial x_{i}}\left(u^{ij}-u^{ji}\right),\,1\leq j\leq d.
\]
By passing to the limit, $f^{j}=\Delta N\left(f^{j}\right)=\sum_{i}\partial_{i}u^{ij}=\left(\mathcal{\tilde{S}}f\right)^{j}+\left(\mathcal{G}f\right)^{j}.$
Therefore, $\mathcal{S}f=\mathcal{\tilde{S}}f$ and 
\begin{align*}
\left(\mathcal{S}f\right)^{j} & =\sum_{i}\frac{\partial}{\partial x_{i}}\left(\partial_{i}\left(N\left(f^{j}\right)\right)-\partial_{j}\left(N\left(f^{i}\right)\right)\right).
\end{align*}
In particular, if $f=\nabla p$ where $p\in H_{\theta+l,p}^{l+1}\left(\mathbf{R}^{d}\right)$
is a scalar, then $\mathcal{S}f=0.$ 
\end{rem}

\section{Schauder Ring Properties}

We now establish slightly modified Schauder ring properties. All parameters
are assumed to be non-negative. 
\begin{lem}
\label{lem:schauringl}(cf. \citep[Theorem 4.39]{AF}) Let $p>1,\,l>\frac{d}{p},\,m+m^{\prime}\leq l.$
Then there exists $C>0$ such that for all $u\in H_{p}^{l-m}\left(\mathbf{R}^{d}\right),v\in H_{p}^{l-m^{\prime}}\left(\mathbf{R}^{d}\right),$

\[
\int\left|u\left(x\right)v\left(x\right)\right|{}^{p}dx\leq C\left|u\right|{}_{H_{p}^{l-m}}^{p}\left|v\right|{}_{H_{p}^{l-m^{\prime}}}^{p}.
\]
\end{lem}
\begin{proof}
\noindent We follow the proof of \citep[Theorem 4.39]{AF}. According
to which, we have the following embedding.

\noindent (i) Let $lp\leq d$ and $p\leq r\leq\frac{dp}{d-lp}$ (or
$p\leq r<\infty$ if $lp=d.$) Then there is a constant $C$ such
that for all $g\in H_{p}^{l}\left(\mathbf{R}^{d}\right),$
\[
\int\left|g\left(x\right)\right|{}^{r}dx\leq C\left|g\right|{}_{H_{p}^{l}}^{r}.
\]

\noindent (ii) Let $lp>d$. Then there is a constant $C$ such that
for all $g\in H_{p}^{l}\left(\mathbf{R}^{d}\right)$ and $dx$-a.s.,
\[
\left|g\left(x\right)\right|\leq C\left|g\right|{}_{H_{p}^{l}}.
\]

\noindent Hence, if $\left(l-m\right)p>d$, then
\[
\int\left|u\left(x\right)v\left(x\right)\right|{}^{p}dx\leq C\left|u\right|{}_{H_{p}^{l-m}}^{p}\left|v\right|{}_{p}^{p}.
\]

\noindent Similarly, if $\left(l-m^{\prime}\right)p>d$, then
\[
\int\left|u\left(x\right)v\left(x\right)\right|{}^{p}dx\leq C\left|u\right|{}_{p}^{p}\left|v\right|{}_{H_{p}^{l-m^{\prime}}}^{p}.
\]
If both $\left(l-m\right)p\leq d$ and $\left(l-m^{\prime}\right)p\leq d$,
noting that we always have $\left(l-m+l-m^{\prime}\right)p>d$, then
\[
\frac{d-\left(l-m\right)p}{d}+\frac{d-\left(l-m^{\prime}\right)p}{d}<1,
\]
and there exist positive numbers $r,r^{\prime}>1$ such that $\frac{1}{r}+\frac{1}{r^{\prime}}=1$,
\begin{eqnarray*}
p & \leq & rp<\frac{dp}{d-\left(l-m\right)p},\\
p & \leq & r^{\prime}p<\frac{dp}{d-\left(l-m^{\prime}\right)p}.
\end{eqnarray*}
Finally, we obtain by applying (i),
\begin{eqnarray*}
\int\left|u\left(x\right)v\left(x\right)\right|{}^{p}dx & \leq & \left(\int\left|u\left(x\right)\right|{}^{rp}dx\right){}^{1/r}\left(\int\left|v\left(x\right)\right|{}^{r^{\prime}p}dx\right){}^{1/r^{\prime}}\\
 & \leq & C\left|u\right|{}_{H_{p}^{l-m}}^{p}\left|v\right|{}_{H_{p}^{l-m^{\prime}}}^{p}.
\end{eqnarray*}

\noindent The proof is complete.
\end{proof}
\begin{cor}
\label{cor-product}Let $p>1,\,l>\frac{d}{p},\,k=k_{1}+\ldots+k_{N}\leq l.$
Then there exists $C>0$ such that for all $u_{i}\in H_{p}^{l-k_{i}}\left(\mathbf{R}^{d}\right),i=1,...,N,$
the product $\prod_{i=i}^{N}u_{i}\in H_{p}^{l-k}\left(\mathbf{R}^{d}\right)$
and
\[
\left|\prod_{i=i}^{N}u_{i}\right|_{H_{p}^{l-k}}\leq C\prod_{i=1}^{N}\left|u_{i}\right|{}_{H_{p}^{l-k_{i}}}.
\]
\end{cor}
\begin{proof}
Let $N=2,$ and $\mu,\mu^{\prime}\in\mathbf{N}_{0}^{d}$ be multi-indices
so that $\left\vert \mu\right\vert +\left\vert \mu^{\prime}\right\vert \leq l-k$.
Since $D^{\mu}u_{1}\in H_{p}^{l-k_{1}-\left\vert \mu\right\vert }\left(\mathbf{R}^{d}\right),D^{\mu^{\prime}}u_{2}\in H_{p}^{l-k_{2}-|\mu^{\prime}|}\left(\mathbf{R}^{d}\right)$
and $k_{1}+\left|\mu\right|+k_{2}+|\mu^{\prime}|\leq l$ , we have
by Lemma \ref{lem:schauringl},
\[
\left|D^{\mu}u_{1}D^{\mu^{\prime}}u_{2}\right|{}_{p}\leq C\left|D^{\mu}u_{1}\right|{}_{H_{p}^{l-k_{1}-|\mu|}}\left|D^{\mu^{\prime}}u_{2}\right|{}_{H_{p}^{l-k_{2}-|\mu^{\prime}|}}\leq C\left|u_{1}\right|{}_{H_{p}^{l-k_{1}}}\left|u_{2}\right|{}_{H_{p}^{l-k_{2}}}.
\]
The statement follows by induction.
\end{proof}
Corollary \ref{cor-product} can be generalized to weighted Sobolev
spaces.
\begin{cor}
\label{cor3:product_w}Let $p>1,\,l>\frac{d}{p},\,k=k_{1}+\ldots+k_{N}\leq l$
and $\delta\leq\delta_{1}+\ldots+\delta_{N}-\left(N-1\right)\frac{d}{p}$.
Then there exists $C>0$ such that for all $u_{i}\in H_{\delta_{i}+l,p}^{l-k_{i}}\left(\mathbf{R}^{d}\right),i=1,\ldots,N,$
the product $\prod_{i=i}^{N}u_{i}\in H_{\delta+l,p}^{l-k}\left(\mathbf{R}^{d}\right)$
and
\[
\left|\prod_{i=i}^{N}u_{i}\right|_{H_{\delta+l,p}^{l-k}}\leq C\prod_{i=1}^{N}\left|u_{i}\right|{}_{H_{\delta_{i}+l,p}^{l-k_{i}}}.
\]
In particular, if $\delta_{1}=\ldots=\delta_{N}=\delta\geq\frac{d}{p}$,
then $\delta_{1}+\ldots+\delta_{N}-\left(N-1\right)\frac{d}{p}=\delta+\left(N-1\right)\left(\delta-\frac{d}{p}\right)\geq\delta$
and hence,
\[
\left|\prod_{i=i}^{N}u_{i}\right|_{H_{\delta+l,p}^{l-k}}\leq C\prod_{i=1}^{N}\left|u_{i}\right|{}_{H_{\delta+l,p}^{l-k_{i}}}.
\]
If, in addition, $k_{i}=0$ for all $i$, then
\[
\left|\prod_{i=i}^{N}u_{i}\right|_{H_{\delta+l,p}^{l}}\leq C\prod_{i=1}^{N}\left|u_{i}\right|{}_{H_{\delta+l,p}^{l}}.
\]
\end{cor}
\begin{proof}
\noindent Indeed, for any multi-index $\gamma\in\mathbf{N}_{0}^{d}$
such that $\left\vert \gamma\right\vert \leq l-k,$ 
\[
w^{\delta+k+\left|\gamma\right|-d/p}D^{\gamma}\left(\prod_{i=1}^{N}u_{i}\right)=w^{\delta-\left(\sum_{i}\delta_{i}-\left(N-1\right)d/p\right)}\sum_{\mu_{1}+\ldots+\mu_{N}=\gamma}\prod_{i=1}^{N}w^{\delta_{i}+k_{i}+\left\vert \mu_{i}\right\vert -d/p}D^{\mu_{i}}u_{i}.
\]

\noindent Since $D^{\mu_{i}}u_{i}\in H_{\delta_{i}+l,p}^{l-k_{i}-\left|\mu_{i}\right|}\left(\mathbf{R}^{d}\right)$
, we have $w^{\delta_{i}+k_{i}+\left\vert \mu_{i}\right\vert -d/p}D^{\mu_{i}}u_{i}\in H_{p}^{l-k_{i}-\left|\mu_{i}\right|}\left(\mathbf{R}^{d}\right)$
and
\[
\left|w^{\delta_{i}+k_{i}+\left|\mu_{i}\right|-d/p}D^{\mu_{i}}u_{i}\right|{}_{H_{p}^{l-k_{i}-\mu_{i}}}\leq C\left|D^{\mu_{i}}u_{i}\right|{}_{H_{\delta_{i}+l,p}^{l-k_{i}-\left|\mu_{i}\right|}}\leq C\left|u_{i}\right|{}_{H_{\delta_{i}+l,p}^{l-k_{i}}}.
\]
The statement follows by Corollary \ref{cor-product}.
\end{proof}
Clearly, Corollary \ref{cor-product} and Corollary \ref{cor3:product_w}
can be extended to multiplication between matrices. We have the following
estimate of function composition. We note that all derivatives are
interpreted as classical since $p>d.$
\begin{lem}
\noindent\label{lem:compose} Let $p>d,\,l\geq0,$$\,\delta,\delta^{\prime}\geq\frac{d}{p},\,N>0.$
Then there exists $C>0$ such that for all $f\in H_{\delta+l,p}^{l}\left(\mathbf{R}^{d}\right)$;
and $g:\mathbf{R}^{d}\rightarrow\mathbf{R}^{d}$ a diffeomorphism
with $\left|\det\nabla\left(g^{-1}\right)\right|_{\infty}+\left\Vert g^{-1}\right\Vert +\chi_{l\geq1}\left|\nabla g\right|_{\infty}+\chi_{l\geq2}\left|D^{2}g\right|_{H_{\delta^{\prime}+l-1,p}^{l-2}}\leq N,$
\begin{align*}
 & \left|f\circ g\right|_{H_{\delta+l,p}^{l}}\\
 & \leq C\left(1+\left\Vert g^{-1}\right\Vert \right)^{\delta-d/p+l}\left|f\right|_{H_{\delta+l,p}^{l}}\left(1+\chi_{l\geq1}\left|\nabla g\right|_{\infty}+\chi_{l\geq2}\left|D^{2}g\right|_{H_{\delta^{\prime}+l-1,p}^{l-2}}\right)^{l}.
\end{align*}
\end{lem}
\begin{proof}
\noindent We first mention a simple estimate,
\[
\frac{w\left(g^{-1}\left(x\right)\right)}{w\left(x\right)}=\frac{\left(1+\left|g^{-1}\left(x\right)\right|^{2}\right){}^{1/2}}{\left(1+\left|x\right|^{2}\right)^{1/2}}\leq1+\left\Vert g^{-1}\right\Vert ,\,\left\Vert g^{-1}\right\Vert =\left\vert g^{-1}\left(0\right)\right\vert +\left\vert \nabla g^{-1}\right\vert _{\infty}.
\]

\noindent If $l=0,$ by changing the variable of integration, 
\[
\left|f\circ g\right|_{H_{\delta,p}^{0}}=\left|w^{\delta-d/p}f\left(g\right)\right|_{p}\leq\left(1+\left\Vert g^{-1}\right\Vert \right)^{\delta-d/p}\left|w^{\delta-d/p}f\right|_{p}.
\]

\noindent For a multi-index $\gamma\in\mathbf{N}_{0}^{d}$ with $\left|\gamma\right|=l\geq1$,
$w^{\delta-d/p+l}D^{\gamma}\left(f\circ g\right)$ is a summation
of terms in the form of
\begin{align*}
\mathcal{A}= & w^{\delta-d/p+m}D^{\mu}f\left(g\right)\prod_{i=1}^{m}w^{\left|\mu_{i}\right|-1}D^{\mu_{i}}g^{a_{i}}
\end{align*}

\noindent where $\mu_{1}+\ldots+\mu_{m}=\gamma,\left|\mu_{i}\right|\geq1\text{\text{ for }}i=1,...,m,\left\vert \mu\right\vert =m,1\leq m\leq l$,
and $1\leq a_{i}\leq d$ are component indices. 

\noindent If $\left|\mu_{i}\right|=1$ for all $i=1,...,m$ then 
\[
\left|\mathcal{A}\right|_{p}\leq C\left(1+\left\Vert g^{-1}\right\Vert \right)^{\delta-d/p+m}\left|w^{\delta-d/p+m}D^{\mu}f\right|_{p}\left|\nabla g\right|_{\infty}^{m}.
\]

\noindent We may now assume that $l\geq2$ and $\left|\mu_{i}\right|\geq2$
for some $i.$ Since $\left(l-m\right)p>d$, by Sobolev embedding
theorem for $f$ and Corollary \ref{cor-product}, 
\begin{align*}
\left|\mathcal{A}\right|_{p} & \leq C\left(1+\left\Vert g^{-1}\right\Vert \right)^{\delta-d/p+m}\left|w^{\delta-d/p+m}D^{\mu}f\right|_{\infty}\prod_{i=1,\left|\mu_{i}\right|=1}^{m}\left|\nabla g^{a_{i}}\right|_{\infty}\left|\prod_{i=1,\left|\mu_{i}\right|\geq2}^{m}w^{\left|\mu_{i}\right|-1}D^{\mu_{i}}g^{a_{i}}\right|_{p}\\
 & \leq C\left(1+\left\Vert g^{-1}\right\Vert \right)^{\delta-d/p+m}\left|f\right|_{H_{\delta+l,p}^{l}}\prod_{i=1,\left|\mu_{i}\right|=1}^{m}\left|\nabla g^{a_{i}}\right|_{\infty}\prod_{i=1,\left|\mu_{i}\right|\geq2}^{m}\left|D^{\mu_{i}}g^{a_{i}}\right|_{H_{d/p+l-1,p}^{l-\left|\mu_{i}\right|}}.
\end{align*}

\noindent The proof is complete. 
\end{proof}
Examining the proof above the restriction $\delta\geq\frac{d}{p}$
can be easily relaxed as follows:
\begin{cor}
\label{cor:compose-neg} Let $p>d,\,l\geq0,$$\,\delta\geq0,\,\delta^{\prime}\geq\frac{d}{p},\,N>0.$
Then there exists $C>0$ such that for all $f\in H_{\delta+l,p}^{l}\left(\mathbf{R}^{d}\right)$
; and $g:\mathbf{R}^{d}\rightarrow\mathbf{R}^{d}$ a diffeomorphism
with $\left|\det\nabla\left(g^{-1}\right)\right|_{\infty}+\left\Vert g\right\Vert +\left\Vert g^{-1}\right\Vert +\chi_{l\geq1}\left|\nabla g\right|_{\infty}+\chi_{l\geq2}\left|D^{2}g\right|_{H_{\delta^{\prime}+l-1,p}^{l-2}}\leq N,$
\begin{align*}
 & \left|f\circ g\right|_{H_{\delta+l,p}^{l}}\\
 & \leq C\left(1+\left\Vert g\right\Vert +\left\Vert g^{-1}\right\Vert \right)^{\left|\delta-d/p\right|+l}\left|f\right|_{H_{\delta+l,p}^{l}}\left(1+\chi_{l\geq1}\left|\nabla g\right|_{\infty}+\chi_{l\geq2}\left|D^{2}g\right|_{H_{\delta^{\prime}+l-1,p}^{l-2}}\right)^{l}.
\end{align*}
\end{cor}
Clearly, Lemma \ref{lem:compose} and Corollary \ref{cor:compose-neg}
also hold if $f$ is a vector or a matrix.

\section{Flow Representation and Main Results}

In this section, we discuss flow representation of the Navier-Stokes
equations and state our main results. We will assume that the prescribed
fields $\mathbf{u}_{0}$ and $G$ satisfy the following assumption
with $l\geq2.$

\noindent$\textbf{Assumption }F\left(l\right).$

\noindent (i) $\mathbf{u}_{0}$ is $\mathcal{F}_{0}^{W}-$measurable.
For all $\omega\in\Omega$, $\mathbf{u}_{0}$ is divergence-free,
and $\mathbf{u}_{0}\in C^{2}\left(\mathbf{R}^{d};\mathbf{R}^{d}\right)\cap H_{\theta+l-2,p}^{l-2}\left(\mathbf{R}^{d};\mathbf{R}^{d}\right).$

\noindent (ii) $G$ is $\mathbb{F}^{W}-$adapted. For all $\left(\omega,t\right)\in\Omega\times\left[0,\infty\right)$,
$G\left(t\right)$ is divergence-free and $G\left(t\right)\in H_{\theta+l-2,p}^{l-2}\left(\mathbf{R}^{d};\mathbf{R}^{d}\right)$.
For all $\omega\in\Omega$, $G\in C\left(\left[0,\infty\right),C^{2}\left(\mathbf{R}^{d};\mathbf{R}^{d}\right)\right).$ 

\subsection{Flow Representation}

Let $\left(\Omega,\mathcal{F},\mathbf{P}\right)$ be a complete probability
space with a filtration $\mathbb{F}$ of $\sigma$\textminus algebras
$\left(\mathcal{F}_{t},t\geq0\right)$ satisfying the usual conditions
and $B_{t},W_{t}$ be independent standard $d-$dimensional Wiener
processes on $\left(\Omega,\mathcal{F},\mathbf{P}\right)$. Let $\mathbb{F}^{W}=\left(\mathcal{F}_{t}^{W},t\geq0\right)$
be the standard sub-filtration of $\mathbb{F}$ generated by $W_{t}$.
We consider the following Navier-Stokes equations with a random forcing
term $G$, 
\begin{align}
d\mathbf{u}\left(t\right) & =\left[\mathcal{S}\left(-u^{k}\left(t\right)\partial_{k}\mathbf{u}\left(t\right)\right)+\frac{\varepsilon^{2}}{2}\Delta\mathbf{u}\left(t\right)+G\left(t\right)\right]dt\label{eq:NSE}\\
\mathbf{u}\left(0\right) & =\mathbf{u_{0}},\hspace{1em}\text{div }\mathbf{u}\left(t\right)=0,\hspace{1em}t>0.\nonumber 
\end{align}

\noindent We now formulate the flow representation of (\ref{eq:NSE}).
Note that $\mathbf{u}_{0}$ and $G$ have sufficient regularity thanks
to Assumption \textbf{$F\left(l\right)$} to make the computation
rigorous. To make the statement less cumbersome, we define some notations
for the next Lemma. For a smooth vector field $\mathbf{u}$, we let
the perturbed flow $\eta\left(t\right):\mathbf{R}^{d}\rightarrow\mathbf{R}^{d},\,t>0$
be given by
\[
d\eta\left(t\right)=\mathbf{u}\left(t,\eta\left(t\right)\right)dt+\varepsilon B_{t},\,\mathbf{\eta}\left(0\right)=e,\,t>0.
\]
and let $\kappa\left(t\right)=\eta^{-1}\left(t\right)$ be its spatial
inverse whenever it is well-defined. Also, we let
\[
\mathbf{g}_{\eta}\left(t,x\right)=\mathbf{u}_{0}+\int_{0}^{t}\left(\nabla\eta\left(s,x\right)\right)^{\ast}G\left(s,\eta\left(s,x\right)\right)ds,\,\left(\omega,t,x\right)\in\Omega\times\left(0,\infty\right)\times\mathbf{R}^{d}.
\]

\begin{lem}
\noindent\label{lem:flow1} Let $p>d,\,l\geq2,\,\theta\in\left(1,d\right),\,\alpha\in\left(0,1\right],\,N_{1},N_{2}>0$
and $F\left(l\right)$ holds. Let $\mathbf{u}:\Omega\times\left[0,T\right]\times\mathbf{R}^{d}\rightarrow\mathbf{R}^{d}$
be $\mathbb{F}^{W}-$adapted such that 
\[
\sup_{\left(\omega,t\right)\in\Omega\times\left[0,T\right]}\left|\mathbf{u}\left(t\right)\right|_{\mathcal{C}^{3+\alpha}}\leq N_{1}.
\]
 If (i) $\left(\nabla\mathbf{\kappa}\left(t\right)\right)^{\ast}\mathbf{g}_{\eta}\left(t,\mathbf{\kappa}\left(t\right)\right)\in H_{\theta+l,p}^{l}\left(\mathbf{R}^{d};\mathbf{R}^{d}\right)$
for all $\left(\omega,t\right)\in\Omega\times\left[0,T\right]$ and 

\noindent (ii) $\mathbf{E}\left|\mathcal{S}\left[\left(\nabla\mathbf{\kappa}\left(t\right)\right)^{\ast}\mathbf{g}_{\eta}\left(t,\mathbf{\kappa}\left(t\right)\right)\right]\right|_{C^{2}}\leq N_{2}$
for all $t\in\left[0,T\right],$

\vspace{10pt}

\noindent then for $P-$a.s., $\mathbf{y}\left(t\right)=\mathbf{E}_{t}^{W}\mathcal{S}\left[\left(\nabla\mathbf{\kappa}\left(t\right)\right)^{\ast}\mathbf{g}_{\eta}\left(t,\mathbf{\kappa}\left(t\right)\right)\right]$
solves 
\begin{align}
d\mathbf{y}\left(t\right)= & \mathcal{S}\left[-u^{k}\left(t\right)\partial_{k}\mathbf{y}\left(t\right)-\left(\nabla\mathbf{u}\left(t\right)\right)^{\ast}\mathbf{y}\left(t\right)\right]dt+\left[\frac{\varepsilon^{2}}{2}\Delta\mathbf{y}\left(t\right)+G\left(t\right)\right]dt\label{eq:LinSNSE}\\
\mathbf{y}\left(0\right)= & \mathbf{u}_{0},\hspace{1em}\text{div }\mathbf{y}\left(t\right)=0,\hspace{1em}t\in\left(0,T\right]\nonumber 
\end{align}

\noindent as an equality in $H_{\theta+l-2,p}^{l-2}\left(\mathbf{R}^{d};\mathbf{R}^{d}\right).$
\end{lem}
\begin{rem}
We will justify in the proof that $\kappa$ is well-defined. For clarity,
we note that $\mathbf{u}_{0}$ satisfies both the above conditions
and assumption $F\left(l\right)$.
\end{rem}
\begin{proof}
\noindent We consider the perturbed flow $\eta\left(t\right):\mathbf{R}^{d}\rightarrow\mathbf{R}^{d}$
given by
\begin{align}
d\eta\left(t\right) & =\mathbf{u}\left(t,\eta\left(t\right)\right)dt+\varepsilon B_{t}\label{eq:d2}\\
\eta\left(0\right) & =e,\hspace{1em}t\in\left(0,T\right].\nonumber 
\end{align}

\noindent According to \citep[Theorem 2.1 and 2.4]{LM}, the classical
solution $\eta\left(t\right)$ of (\ref{eq:d2}) and its spatial inverse
$\mathbf{\kappa}\left(t\right)=\mathbf{\eta}^{-1}\left(t\right)$
belongs to $C\left(\left[0,T\right],C^{3}\left(\mathbf{R}^{d};\mathbf{R}^{d}\right)\right)$
and for $P-$a.s., $\kappa\left(t\right)$ is the classical solution
of the following equation, 
\begin{eqnarray*}
d\kappa\left(t\right) & = & \left[-\nabla\kappa\left(t\right)\mathbf{u}\left(t\right)+\frac{\varepsilon^{2}}{2}\Delta\mathbf{\kappa}\left(t\right)\right]dt-\varepsilon\partial_{k}\mathbf{\kappa}\left(t\right)dB_{t}^{k}\\
\mathbf{\kappa}\left(0\right) & = & e,\hspace{1em}t\in\left(0,T\right].
\end{eqnarray*}
Therefore, for any $m,l=1,\ldots,d,$
\begin{align*}
 & d\partial_{l}\kappa^{m}\left(t\right)\\
 & =\left[-\nabla\partial_{l}\kappa^{m}\left(t\right)\mathbf{u}\left(t\right)-\nabla\kappa^{m}\left(t\right)\partial_{l}\mathbf{u}\left(t\right)+\frac{\varepsilon^{2}}{2}\Delta\partial_{l}\kappa^{m}\left(t\right)\right]dt-\varepsilon\partial_{k}\partial_{l}\kappa^{m}\left(t\right)dB_{t}^{k}.
\end{align*}

\noindent By writing the above equation in a matrix form with $A\left(t\right)=\left(\partial_{l}\kappa^{m}\left(t\right)\right){}_{\substack{1\leq m\leq d,\\
1\leq l\leq d
}
}$ and multiplying by $\mathbf{g}_{\eta}\left(t,\kappa\left(t\right)\right)$
from the right, we obtain

\noindent
\begin{align*}
 & \left(dA\left(t\right)^{\ast}\right)\mathbf{g}_{\eta}\left(t,\kappa\left(t\right)\right)\\
 & =\left(-u^{k}\left(t\right)\partial_{k}A\left(t\right)^{\ast}-\left(\nabla\mathbf{u}\left(t\right)\right)^{\ast}A\left(t\right)^{\ast}+\frac{\varepsilon^{2}}{2}\Delta A\left(t\right)^{\ast}\right)\mathbf{g}_{\eta}\left(t,\kappa\left(t\right)\right)dt\\
 & -\varepsilon\partial_{k}A\left(t\right)^{\ast}\mathbf{g}_{\eta}\left(t,\kappa\left(t\right)\right)dB_{t}^{k}.
\end{align*}

\noindent Next applying Itô-Wentzell formula for $\mathbf{g}_{\eta}\left(t,\kappa\left(t\right)\right)$
and multiplying by $A\left(t\right)^{\ast}$ from the left, we obtain

\noindent
\begin{align*}
 & A\left(t\right)^{\ast}d\mathbf{g}_{\eta}\left(t,\kappa\left(t\right)\right)\\
 & =\left[G\left(t\right)-A\left(t\right)^{\ast}\nabla\mathbf{g}_{\eta}\left(t,\kappa\left(t\right)\right)\nabla\kappa\left(t\right)\mathbf{u}\left(t\right)\right]dt\\
 & +\frac{\varepsilon^{2}}{2}\left(A\left(t\right)^{\ast}\nabla\mathbf{g}_{\eta}\left(t,\kappa\left(t\right)\right)\Delta\mathbf{\kappa}\left(t\right)+A\left(t\right)^{\ast}\partial_{ij}^{2}\mathbf{g}_{\eta}\left(t,\kappa\left(t\right)\right)\left(\nabla\kappa^{i}\left(t\right)\cdot\nabla\kappa^{j}\left(t\right)\right)\right)dt\\
 & -\varepsilon A\left(t\right)^{\ast}\nabla\mathbf{g}_{\eta}\left(t,\kappa\left(t\right)\right)\partial_{k}\mathbf{\kappa}\left(t\right)dB_{t}^{k}.
\end{align*}

\noindent Finally, the covariation term is $d\left[A\left(t\right)^{\ast},\mathbf{g}_{\eta}\left(t,\kappa\left(t\right)\right)\right]=\epsilon^{2}\partial_{k}A\left(t\right)^{\ast}\nabla\mathbf{g}_{\eta}\left(t,\kappa\left(t\right)\right)\partial_{k}\mathbf{\kappa}\left(t\right)dt.$

\noindent By Itô product rule, summing the terms above, $\mathbf{z}\left(t\right)=A\left(t\right)^{\ast}\mathbf{g}_{\eta}\left(t,\mathbf{\kappa}\left(t\right)\right)$
must satisfy the following equation
\begin{eqnarray*}
d\mathbf{z}\left(t\right) & = & \left[G\left(t\right)-u^{k}\left(t\right)\partial_{k}\mathbf{z}\left(t\right)-\left(\nabla\mathbf{u}\left(t\right)\right)^{\ast}\mathbf{z}\left(t\right)+\frac{\varepsilon^{2}}{2}\Delta\mathbf{z}\left(t\right)\right]dt\\
 & - & \varepsilon\partial_{k}\mathbf{z}\left(t\right)dB_{t}^{k}\\
\mathbf{z}\left(0\right) & = & \mathbf{u}_{0}.
\end{eqnarray*}

\noindent Due to (i), $\mathbf{z}\left(t\right)\in H_{\theta+l,p}^{l}\left(\mathbf{R}^{d};\mathbf{R}^{d}\right)$.
According to Lemma \ref{lem:mik1_smoothing_decomp}, we may let $\mathbf{z}\left(t\right)=\mathcal{S}\mathbf{z}\left(t\right)+\left(\nabla p\left(t\right)\right)^{\ast}$
be its Helmholtz decomposition in $H_{\theta+l,p}^{l}\left(\mathbf{R}^{d};\mathbf{R}^{d}\right)$
where $p\left(t\right)\in H_{\theta+l}^{l+1}\left(\mathbf{R}^{d}\right)$
is a scalar. By collecting all gradient terms, we derive 
\begin{eqnarray*}
d\mathbf{z}\left(t\right) & = & \left[G\left(t\right)-u^{k}\left(t\right)\partial_{k}\mathcal{S}\mathbf{z}\left(t\right)-\left(\nabla\mathbf{u}\left(t\right)\right)^{\ast}\mathcal{S}\mathbf{z}\left(t\right)+\frac{\varepsilon^{2}}{2}\Delta\mathcal{S}\mathbf{z}\left(t\right)+\left(\nabla q\left(t\right)\right)^{\ast}\right]dt\\
 & - & \varepsilon\partial_{k}\mathbf{z}\left(t\right)dB_{t}^{k}\\
\mathbf{z}\left(0\right) & = & \mathbf{u}_{0}
\end{eqnarray*}

\noindent where $q\left(t\right)=-\nabla p\left(t\right)\mathbf{u}\left(t\right)-\frac{\varepsilon^{2}}{2}\Delta p\left(t\right).$
Next, we take the optional projection $\mathbf{E}_{t}^{W}$on both
sides. Due to (ii), $\mathbf{E}_{t}^{W}$ can be interchanged with
the integral with respect to $dt$ and derivatives of $\mathcal{S}\mathbf{z}\left(t\right).$ 

\noindent Finally, by taking a solenoidal projection in $H_{\theta+l-2,p}^{l-2}\left(\mathbf{R}^{d};\mathbf{R}^{d}\right)$
and applying Remark \ref{rem:sol_0}, it is easily verified that 

\noindent
\begin{eqnarray*}
\mathbf{y}\left(t\right) & =\mathbf{E}_{t}^{W}\mathcal{S}\mathbf{z}\left(t\right)= & \mathbf{E}_{t}^{W}\mathcal{S}\left[\left(\nabla\mathbf{\kappa}\left(t\right)\right)^{\ast}\mathbf{g_{\eta}}\left(t,\mathbf{\kappa}\left(t\right)\right)\right]
\end{eqnarray*}

\noindent is a solution of (\ref{eq:LinSNSE}) as an equality in $H_{\theta+l-2,p}^{l-2}\left(\mathbf{R}^{d};\mathbf{R}^{d}\right)$.
\end{proof}
In the next Lemma, we derive a simplified form of $\mathbf{E}_{t}^{W}\mathcal{S}\left[\left(\nabla\mathbf{\kappa}\left(t\right)\right)^{\ast}\mathbf{g_{\eta}}\left(t,\mathbf{\kappa}\left(t\right)\right)\right].$
\begin{lem}
\label{lem:simplified_K}Let $p>d,\,l\geq2,\,\theta\in\left(1,d\right)$
and $F\left(l\right)$ holds. Let $\eta\left(t\right):\mathbf{R}^{d}\rightarrow\mathbf{R}^{d},\,t\in\left[0,T\right]$
be a diffeomorphism and $\kappa\left(t\right):\mathbf{R}^{d}\rightarrow\mathbf{R}^{d},\,t\in\left[0,T\right]$
be its spatial inverse such that 

\vspace{10pt}

\noindent (i) $\left(\nabla\mathbf{\kappa}\left(t\right)\right)^{\ast}\mathbf{g_{\eta}}\left(t,\mathbf{\kappa}\left(t\right)\right)\in H_{\theta+l,p}^{l}\left(\mathbf{R}^{d};\mathbf{R}^{d}\right)$
for all $\left(\omega,t\right)\in\Omega\times\left[0,T\right],$ 

\noindent (ii) $\eta\left(t\right)$, $\kappa\left(t\right)$ are
spatially twice differentiable for all $\left(\omega,t\right)\in\Omega\times\left[0,T\right],$ 

\noindent (iii) for any multi-index $\gamma\in\mathbf{N}_{0}^{d}$
with $0\leq\left|\gamma\right|\leq2$, $D^{\gamma}\eta\left(t,x\right)$
is continuous in $t$ for all $\left(\omega,x\right)\in\Omega\times\mathbf{R}^{d}.$

\noindent Then for all $\left(\omega,t\right)\in\Omega\times\left[0,T\right],$
\begin{equation}
\mathcal{S}\left(\left(\nabla\mathbf{\kappa}\left(t\right)\right)^{\ast}\mathbf{g_{\eta}}\left(t,\mathbf{\kappa}\left(t\right)\right)\right)=K_{\eta,\mathbf{h_{\eta}}}\left(t\right)\label{eq:S2K}
\end{equation}

\noindent where for $j=1,...,d,$
\begin{align}
K_{\eta,\mathbf{h_{\eta}}}^{j}\left(t,x\right) & =\int\Gamma_{i}\left(x-z\right)\left[\phi_{\eta,\mathbf{h_{\eta}}}^{ji}\left(t,z\right)-\phi_{\eta,\mathbf{h}_{\eta}}^{ij}\left(t,z\right)\right]dz,\label{eq:SimplifiedKform}
\end{align}

\noindent
\[
\phi_{\eta,\mathbf{h}}\left(t,x\right)=\left(\nabla\kappa\left(t,x\right)\right)^{\ast}\mathbf{h}\left(t,\kappa\left(t,x\right)\right)\nabla\kappa\left(t,x\right)
\]

\noindent and
\[
\mathbf{h}_{\eta}\left(t,x\right)=\nabla\mathbf{u_{0}}\left(x\right)+\int_{0}^{t}\left(\nabla\eta\left(s,x\right)\right)^{\ast}\nabla G\left(s,\eta\left(s,x\right)\right)\left(\nabla\eta\left(s,x\right)\right)ds,
\]

\noindent for all $\left(\omega,t,x\right)\in\Omega\times\left[0,T\right]\times\mathbf{R}^{d}.$
\end{lem}
\begin{proof}
\noindent In fact, if $f\in H_{\theta+l,p}^{l}\left(\mathbf{R}^{d};\mathbf{R}^{d}\right)$
then by Remark \ref{rem:sol_0},
\begin{equation}
\left(\mathcal{S}f\right)^{j}=\int\Gamma_{i}\left(\cdot-z\right)\left(\partial_{i}f^{j}\left(z\right)-\partial_{j}f^{i}\left(z\right)\right)dz.\label{eq:SF}
\end{equation}

\noindent Applying (\ref{eq:SF}) for $\left(\nabla\mathbf{\kappa}\left(t\right)\right)^{\ast}\mathbf{g_{\eta}}\left(t,\mathbf{\kappa}\left(t\right)\right)\in H_{\theta+l,p}^{l}\left(\mathbf{R}^{d};\mathbf{R}^{d}\right),$
we obtain
\begin{align*}
 & \left(\mathcal{S}\left(\left(\nabla\mathbf{\kappa}\left(t\right)\right)^{\ast}\mathbf{g_{\eta}}\left(t,\mathbf{\kappa}\left(t\right)\right)\right)\right)^{j}\\
 & =\int\Gamma_{i}\left(\cdot-z\right)\left[\partial_{i}\left(\partial_{j}\mathbf{\kappa}^{k}\left(t,z\right)\mathbf{g}_{\eta}^{k}\left(t,\mathbf{\kappa}\left(t,z\right)\right)\right)-\partial_{j}\left(\partial_{i}\mathbf{\kappa}^{k}\left(t,z\right)\mathbf{g}_{\eta}^{k}\left(t,\mathbf{\kappa}\left(t,z\right)\right)\right)\right]dz\\
 & =\int\Gamma_{i}\left(\cdot-z\right)\left[\partial_{j}\mathbf{\kappa}^{k}\left(t,z\right)\nabla\mathbf{g}_{\eta}^{k}\left(t,\mathbf{\kappa}\left(t,z\right)\right)\left(\partial_{i}\kappa\left(t,z\right)\right)^{\ast}-\partial_{i}\mathbf{\kappa}^{k}\left(t,z\right)\nabla\mathbf{g}_{\eta}^{k}\left(t,\mathbf{\kappa}\left(t,z\right)\right)\left(\partial_{j}\kappa\left(t,z\right)\right)^{\ast}\right]dz\\
 & =\int\Gamma_{i}\left(\cdot-z\right)\left[\left(\nabla\mathbf{\kappa}\left(t,z\right)^{\ast}\nabla\mathbf{g}_{\eta}\left(t,\mathbf{\kappa}\left(t,z\right)\right)\nabla\kappa\left(t,z\right)\right)^{ji}-\left(\nabla\mathbf{\kappa}\left(t,z\right)^{\ast}\nabla\mathbf{g}_{\eta}\left(t,\mathbf{\kappa}\left(t,z\right)\right)\nabla\kappa\left(t,z\right)\right)^{ij}\right]dz
\end{align*}
and by (iii),
\begin{align*}
\nabla\mathbf{g}_{\eta}\left(t\right) & =\nabla\mathbf{u}_{0}+\int_{0}^{t}\left(\nabla\eta\left(s\right)\right)^{\ast}\nabla G\left(s,\eta\left(s\right)\right)\nabla\eta\left(s\right)ds+\int_{0}^{t}A\left(s\right)ds
\end{align*}

\noindent where $\left(A\left(s\right)\right)_{ij}=\partial_{ij}^{2}\eta^{k}\left(s\right)G^{k}\left(s,\eta\left(s\right)\right)$
is symmetric. The statement follows.
\end{proof}
\begin{rem}
\noindent\label{rem:formal_flow_derivation}For the time being, we
provide a formal argument to derive the form of the flow equations.
If a diffeomorphism $\eta\left(t\right)$ and $\kappa\left(t\right)=\eta^{-1}\left(t\right)$
its spatial inverse satisfy
\begin{eqnarray}
d\mathbf{\eta}\left(t\right) & = & \mathbf{E}_{t}^{W}\mathcal{S}\left[\left(\nabla\mathbf{\kappa}\left(t,z\right)\right)^{\ast}\mathbf{g}_{\eta}\left(t,\mathbf{\kappa}\left(t,z\right)\right)\right]|_{z=\mathbf{\eta}\left(t\right)}dt+\varepsilon dB_{t}\label{eq:d4}\\
 & = & \mathbf{E}_{t}^{W}K_{\eta,\mathbf{h_{\eta}}}\left(t,\eta\left(t\right)\right)+\epsilon dB_{t}\nonumber \\
\eta\left(0\right) & = & e,\hspace{1em}t\in\left(0,T\right],\nonumber 
\end{eqnarray}
then by letting $\mathbf{u}\left(t\right)=\mathbf{E}_{t}^{W}\mathcal{S}\left[\left(\nabla\mathbf{\kappa}\left(t\right)\right)^{\ast}\mathbf{g}_{\eta}\left(t,\mathbf{\kappa}\left(t\right)\right)\right]$
in Lemma \ref{lem:flow1}, (\ref{eq:LinSNSE}) becomes (\ref{eq:NSE})
where the term $\left(\nabla\mathbf{u}\left(t\right)\right)^{\ast}\mathbf{u}\left(t\right)=\frac{1}{2}\nabla\left|\mathbf{u}\left(t\right)\right|^{2}$
disappears under the solenoidal projection.

Our strategy is to find a solution of (\ref{eq:d4}) in appropriate
weighted Sobolev spaces and then return to Lemma \ref{lem:flow1}
and show that indeed the velocity 
\[
\mathbf{u}\left(t\right)=\mathbf{E}_{t}^{W}\mathcal{S}\left[\left(\nabla\mathbf{\kappa}\left(t\right)\right)^{\ast}\mathbf{g}_{\eta}\left(t,\mathbf{\kappa}\left(t\right)\right)\right]
\]

\noindent is a $H_{\theta,p}^{l+1}-$solution of (\ref{eq:NSE}) where
the equality is understood in $H_{\theta+l-2,p}^{l-2}\left(\mathbf{R}^{d};\mathbf{R}^{d}\right)$
(see Theorem \ref{thm:main_velocity} for the complete statement.) 
\end{rem}

\subsection{Main Results}

We are now ready to state main results of this paper. Intuitively,
$\eta\left(t\right)$ remains close to the identity mapping for a
short time. Due to the lack of Sobolev regularity of constants and
the identity map $e$, it is convenient to consider for $t\geq0$
the displacement $\zeta\left(t\right)$ defined by 

\noindent
\begin{align*}
\zeta\left(t\right) & =\eta\left(t\right)-e-\epsilon B_{t}.
\end{align*}

\noindent Therefore, from (\ref{eq:d4}), $\zeta\left(t\right)$ must
satisfy the equation 

\noindent
\begin{eqnarray*}
d\mathbf{\zeta}\left(t\right) & = & \mathbf{E}_{t}^{W}K_{\eta,\mathbf{h}_{\eta}}\left(t,\eta\left(t\right)\right)dt\\
\zeta\left(0\right) & = & 0.
\end{eqnarray*}

We start with the existence of the flow equation.
\begin{thm}
\textcolor{red}{\label{thm:main_flow} }Let $p>d,\,l>1+\frac{d}{p},\,\theta\in\left[1+\frac{d}{p},d\right),\,N>0$
and 
\[
\sup_{\omega\in\Omega}\mathbf{\left|u_{0}\right|}_{H_{\theta+l,p}^{l+1}}+\sup_{\left(\omega,t\right)\in\Omega\times\left[0,\infty\right)}\left|G\left(t\right)\right|_{H_{\theta+l,p}^{l+1}}\leq N.
\]
Then for some deterministic $T>0,$ there exists $\zeta\in C\left(\left[0,T\right],H_{\theta+l,p}^{l+1}\left(\mathbf{R}^{d};\mathbf{R}^{d}\right)\right)$
such that
\begin{align}
\zeta\left(t\right) & =\int_{0}^{t}\mathbf{E}_{s}^{W}K_{\eta,\mathbf{h}_{\eta}}\left(s,\eta\left(s\right)\right)ds,\,\left(\omega,t\right)\in\Omega\times\left[0,T\right]\label{eq:iter}
\end{align}

\noindent holds in $H_{\theta+l-1,p}^{l}\left(\mathbf{R}^{d};\mathbf{R}^{d}\right)$.
Moreover, there exists $M>0$ such that 
\[
\sup_{\left(\omega,t\right)\in\Omega\times\left[0,T\right]}\left|\zeta\right|_{H_{\theta+l,p}^{l+1}}\text{\ensuremath{\leq}}M.
\]
\end{thm}
Next, we state the existence of the velocity equation. 
\begin{thm}
\noindent\label{thm:main_velocity} Let $p>d,\,l>2+\frac{d}{p},\,\theta\in\left[1+\frac{d}{p},d\right),\,N>0,$
$F\left(l\right)$ holds and 
\[
\sup_{\omega\in\Omega}\mathbf{\left|u_{0}\right|}_{H_{\theta+l,p}^{l+1}}+\sup_{\left(\omega,t\right)\in\Omega\times\left[0,\infty\right)}\left|G\left(t\right)\right|_{H_{\theta+l,p}^{l+1}}\leq N.
\]

\noindent Suppose that $\zeta$ is the solution of (\ref{eq:iter})
as given in Theorem \ref{thm:main_flow}. Then for $P-$a.s., $\mathbf{u}\left(t\right):=\mathbf{E}_{t}^{W}\mathcal{S}\left(\nabla\mathbf{\kappa}\left(t\right)\right)^{\ast}\mathbf{g}_{\eta}\left(t,\mathbf{\kappa}\left(t\right)\right),\,\left(\omega,t\right)\in\Omega\times\left[0,T\right]$
solves
\begin{align}
d\mathbf{u}\left(t\right)= & \left[\mathcal{S}\left(-u^{k}\left(t\right)\partial_{k}\mathbf{u}\left(t\right)\right)+\frac{\varepsilon^{2}}{2}\Delta\mathbf{u}\left(t\right)+G\left(t\right)\right]dt\label{eq:NSE-cor}\\
\mathbf{u}\left(0\right)= & \mathbf{u}_{0},\hspace{1em}\text{div }\mathbf{u}\left(t\right)=0,\hspace{1em}t\in\left(0,T\right]\nonumber 
\end{align}

\noindent as an equality in $H_{\theta+l-2,p}^{l-2}\left(\mathbf{R}^{d};\mathbf{R}^{d}\right).$
Moreover, there exists $M>0$ such that 
\[
\sup_{\left(\omega,t\right)\in\Omega\times\left[0,T\right]}\left|\mathbf{u}\left(t\right)\right|_{H_{\theta+l,p}^{l+1}}\leq M.
\]
\end{thm}

\section{Estimates of Diffeomorphisms}

We collect some basic estimates regarding diffeomorphisms. First,
for any differentiable function $f:\mathbf{R}^{d}\rightarrow\mathbf{R}^{d}$,

\noindent
\[
\frac{w\left(f\left(x\right)\right)}{w\left(x\right)}=\frac{\left(1+\left|f\left(x\right)\right|^{2}\right){}^{1/2}}{\left(1+\left|x\right|^{2}\right)^{1/2}}\leq1+\left\Vert f\right\Vert ,\,\left\Vert f\right\Vert =\left\vert f\left(0\right)\right\vert +\left\vert \nabla f\right\vert _{\infty}.
\]

\noindent If $\eta=e+b+\zeta$, for some $b\in\mathbf{R}^{d}$ and
a continuously differentiable $\zeta:\mathbf{R}^{d}\rightarrow\mathbf{R}^{d}$
with $\left\vert \nabla\zeta\right\vert _{\infty}\leq\frac{1}{2d^{2}}$,
then $\eta$ is a diffeomorphism with spatial inverse $\kappa$ ,
$\nabla\eta=I+\nabla\zeta$, $\left(\nabla\eta\right)^{-1}=\sum_{n=0}^{\infty}\left(-1\right)^{n}\left(\nabla\zeta\right)^{n}$,
and 
\begin{eqnarray}
\left\vert \nabla\kappa\right\vert _{\infty} & = & \left\vert \left(\nabla\eta\right)^{-1}\right\vert _{\infty}\leq1+\sum_{n=1}^{\infty}d^{n-1}\left\vert \nabla\zeta\right\vert _{\infty}^{n}\leq1+\sum_{n=1}^{\infty}\frac{1}{2^{n}d^{n+1}}\label{eq:nablaInvInf}\\
 & \leq & 1+\frac{1}{d\left(2d-1\right)}\leq2.\nonumber 
\end{eqnarray}

\noindent Moreover, 

\noindent
\begin{align}
\left|\nabla\kappa-I\right|_{\infty} & =\left|\left(\nabla\eta\right)^{-1}-I\right|_{\infty}\leq\frac{1}{d\left(2d-1\right)}\leq\frac{1}{2d}.\label{eq:inverseI}
\end{align}

\noindent For any $x,y\in\mathbf{R}^{d},$
\[
\left\vert x-y\right\vert \leq\left\vert \nabla\kappa\right\vert _{\infty}\left\vert \eta\left(x\right)-\eta\left(y\right)\right\vert 
\]
If we take $y=\kappa\left(0\right),x=0$, then 
\begin{equation}
\left\vert \kappa\left(0\right)\right\vert \leq\left\vert \nabla\kappa\right\vert _{\infty}\left\vert \eta\left(0\right)\right\vert .\label{eq:kappa0}
\end{equation}

\noindent Therefore, by (\ref{eq:nablaInvInf}), 
\begin{eqnarray}
\left\Vert \kappa\right\Vert  & = & \left\vert \kappa\left(0\right)\right\vert +\left\vert \nabla\kappa\right\vert _{\infty}\leq\left(1+\left\vert \eta\left(0\right)\right\vert \right)\left\vert \nabla\kappa\right\vert _{\infty}\label{eq:eta_inv_db}\\
 & \leq & 2\left(1+\left\vert b\right\vert +\left\vert \zeta\left(0\right)\right\vert \right).\nonumber 
\end{eqnarray}

\noindent For determinants of Jacobian matrices, by (\ref{eq:nablaInvInf})
for all $x\in\mathbf{R}^{d},$
\begin{eqnarray}
\left\vert \det\nabla\kappa\left(x\right)\right\vert  & \leq & C\left\vert \nabla\kappa\right\vert _{\infty}^{d}\leq C,\label{eq:JacobInv}\\
\left\vert \det\nabla\eta\left(x\right)\right\vert  & \leq & C\left(1+\left\vert \nabla\zeta\right\vert _{\infty}\right)^{d}\leq C.\label{eq:Jacob}
\end{eqnarray}

We now discuss linear combination of $\eta=e+b+\zeta$ and $\bar{\eta}=e+\bar{b}+\bar{\zeta}$
where $\left\vert \nabla\zeta\right\vert _{\infty},\left\vert \nabla\bar{\zeta}\right\vert _{\infty}\leq\frac{1}{2d^{2}}$.
Considering for $s\in\left[0,1\right],\eta_{s}=\left(1-s\right)\eta+s\bar{\eta}$,
$b_{s}=\left(1-s\right)b+s\bar{b},\zeta_{s}=\left(1-s\right)\zeta+s\bar{\zeta},$
we have 
\begin{eqnarray*}
\eta_{s} & = & e+b_{s}+\zeta_{s},\\
\nabla\eta_{s} & = & I+\nabla\zeta_{s},
\end{eqnarray*}

\noindent where $\left|\nabla\zeta_{s}\right|_{\infty}\leq\frac{1}{2d^{2}}.$
Hence, by (\ref{eq:nablaInvInf}), $\left\vert \nabla\left(\eta_{s}^{-1}\right)\right\vert _{\infty}=\left\vert \left(\nabla\eta_{s}\right)^{-1}\right\vert _{\infty}\leq2$.
Denoting $b_{0}=\max\left\{ \left\vert b\right\vert ,\left\vert \bar{b}\right\vert \right\} ,$
$l_{0}=\max\left\{ \left\vert \zeta\left(0\right)\right\vert ,\left\vert \bar{\zeta}\left(0\right)\right\vert \right\} ,$
it follows from (\ref{eq:eta_inv_db}) that

\begin{equation}
\left\Vert \eta_{s}^{-1}\right\Vert \leq2\left(1+b_{0}+l_{0}\right).\label{eq:InterInvD}
\end{equation}

Obviously, for all $x\in\mathbf{R}^{d}$, 
\begin{equation}
\left\vert \det\nabla\left(\eta_{s}^{-1}\right)\left(x\right)\right\vert \leq C,\,\left\vert \det\nabla\eta_{s}\left(x\right)\right\vert \leq C.\label{eq:det_Inter}
\end{equation}

\begin{defn}
For $M>0,\,T>0,\,p>1,\,l\geq0,\,\theta\in\left(1,d\right),$ we say
that $\zeta\in\mathcal{A}_{M,T}^{p,l,\theta}$ if for all $\left(\omega,t\right)\in\Omega\times\left[0,T\right],$
$\zeta\left(t\right)$ is continuously differentiable and 

\noindent (i)$\left\vert \nabla\zeta\left(t\right)\right\vert _{\infty}\leq\frac{1}{2d^{2}}$,
$\left|\zeta\left(t,0\right)\right|\leq M,$

\noindent (ii) $\left|\zeta\left(t\right)\right|_{H_{\theta+l,p}^{l+1}}\leq M.$
\end{defn}
We will consistently denote $\eta\left(t\right)=e+\epsilon B_{t}+\zeta\left(t\right)$
and $\kappa\left(t\right)$ to be the spatial inverse of $\eta\left(t\right).$
Assumptions (i), (ii) will be used without explicit mention. Clearly,
if $\zeta\in\mathcal{A}_{M,T}^{p,l,\theta}$ then $\eta\left(t\right)$
is a diffeomorphism. We start with a simple Lemma which facilitates
later computations. To ease notation, we will write $\lambda_{t}=\lambda_{t}^{\epsilon}=1+\epsilon\sup_{s\in\left[0,t\right]}\left|B_{s}\right|,\,t\geq0.$
\begin{lem}
\noindent\label{lem:weightC} There exists $C>0$ such that for all
$\zeta\in\mathcal{A}_{M,T}^{p,l,\theta}$ and $\left(\omega,t,x\right)\in\Omega\times\left[0,T\right]\times\mathbf{R}^{d},$
\[
\frac{w\left(\eta\left(t\right)\right)}{w}\leq C\lambda_{t},\hspace{1em}\frac{w\left(\kappa\left(t\right)\right)}{w}\leq C\lambda_{t}.
\]

\noindent Moreover, 
\[
\left|\det\nabla\eta\left(t\right)\right|_{\infty}\leq C,\hspace{1em}\left|\det\nabla\kappa\left(t\right)\right|_{\infty}\leq C.
\]

\noindent All estimates also hold for a linear interpolation $\eta_{a}=a\eta_{1}+\left(1-a\right)\eta_{2},\,a\in\left[0,1\right]$
of $\eta_{i}\in\mathcal{A}_{M,T}^{p,l,\theta},\,$$i=1,2$ with constants
independent of $a.$ 
\end{lem}
\begin{proof}
\noindent Trivially, 

\noindent
\begin{align*}
\frac{w\left(\eta\left(t\right)\right)}{w} & \leq1+\left\Vert \eta\left(t\right)\right\Vert =1+\left|\eta\left(t,0\right)\right|+\left|\nabla\eta\left(t\right)\right|_{\infty}\\
 & \leq1+\left|\zeta\left(t,0\right)\right|+\epsilon\left|B_{t}\right|+1+\left|\nabla\zeta\left(t\right)\right|_{\infty}\\
 & \leq C\lambda_{t}.
\end{align*}

\noindent The estimates for $\kappa\left(t\right)=\eta^{-1}\left(t\right)$
follows directly from (\ref{eq:eta_inv_db}). The estimates for Jacobian
matrices follows immediately from (\ref{eq:Jacob}) and (\ref{eq:JacobInv})
respectively.
\end{proof}

\subsection{Growth Estimates of the Spatial Inverse }
\begin{lem}
\label{lem:kappaGrowth} Let $p>d,\,l\geq1,\,\theta\in\left(1,d\right),\,\sigma=\theta-\frac{d}{p}.$ 

\noindent (i) There exists $C>0$ such that for all $\zeta\in\mathcal{A}_{M,T}^{p,l,\theta}$
and $\left(\omega,t\right)\in\Omega\times\left[0,T\right],$ 

\noindent
\begin{align*}
\left|\nabla\kappa\left(t\right)-I\right|_{\infty} & \leq\frac{1}{2d},\\
\left|\kappa\left(t,0\right)\right| & \leq C\lambda_{t}.
\end{align*}

\noindent (ii) There exists $C>0$ such that for all $\zeta\in\mathcal{A}_{M,T}^{p,l,\theta}$
and $\left(\omega,t\right)\in\Omega\times\left[0,T\right],$ 
\[
\left|w^{\sigma-1}\left(\nabla\kappa\left(t\right)-I\right)\right|_{p}\leq C\lambda_{t}^{\left|\sigma-1\right|}.
\]

\noindent (iii) There exists $C>0$ such that for all $\zeta\in\mathcal{A}_{M,T}^{p,l,\theta},$
$\left(\omega,t\right)\in\Omega\times\left[0,T\right]$ and $1\leq k\leq l,$

\noindent
\begin{eqnarray*}
\left\vert w^{k}D^{k}\nabla\kappa\left(t\right)\right\vert _{p} & \leq & C\lambda_{t}^{p_{k}},\\
\left\vert w^{\sigma+k}D^{k}\nabla\kappa\left(t\right)\right\vert _{p} & \leq & C\lambda_{t}^{\sigma+p_{k}},
\end{eqnarray*}

\noindent where $p_{k}$ is defined recursively as $p_{k}=k\left(1+p_{k-1}\right),\,k\geq2,\,p_{1}=1.$
\end{lem}
\begin{proof}
\noindent The variable $t$ is mostly dropped throughout the proof
since all estimates simply holds for each $t$. We resort to tools
in differential calculus on norm vector spaces. For a primer on the
subject, an interested reader may consult \citep{Ca}. 

\noindent (i) The first estimate is simply (\ref{eq:inverseI}). For
the second estimate, by (\ref{eq:kappa0}),
\begin{align*}
\left|\kappa\left(t,0\right)\right| & \leq C\left|\eta\left(t,0\right)\right|\leq C\lambda_{t}.
\end{align*}

\noindent (ii) By changing the variable of integration and Lemma \ref{lem:weightC},

\noindent
\begin{align*}
 & \left|w^{\sigma-1}\left(\nabla\kappa-I\right)\right|_{p}\\
 & \leq\sum_{n=1}^{\infty}\left|w^{\sigma-1}\left(\nabla\zeta\left(\kappa\right)\right)^{n}\right|_{p}\leq C\lambda_{t}^{\left|\sigma-1\right|}\sum_{n=1}^{\infty}\left|\left(\nabla\zeta\right)\right|_{\infty}^{n-1}\left|w^{\sigma-1}\nabla\zeta\right|_{p}\\
 & \leq C\lambda_{t}^{\left|\sigma-1\right|}.
\end{align*}

\noindent (iii) Our goal now is to obtain the form of derivatives
of $\kappa$. Let $M_{d\times d}$ the set of all $d\times d$ matrices
and $U$ be the set of invertible $d\times d$ matrices. Then $F\left(A\right)=A^{-1},A\in U,$
is smooth and its $n-$th Frechet derivative is a continuous multilinear
mapping defined as
\[
F^{\left(n\right)}\left(A\right)\cdot\left(x_{1},...,x_{n}\right)=\left(-1\right)^{n}\sum_{\sigma}A^{-1}x_{\sigma\left(1\right)}A^{-1}...A^{-1}x_{\sigma\left(n\right)}A^{-1}.
\]

\noindent where the summation is taken over all possible permutations.
Clearly, the operator norm of $F^{\left(n\right)}$ satisfies

\noindent
\[
\left\Vert F^{\left(n\right)}\left(A\right)\right\Vert \leq C\left\vert A^{-1}\right\vert ^{n+1},\,A\in U,\,n\geq0
\]
for some $C>0$ (for more details see \citep[Theorem 5.4.3 and relevant exercises]{Ca}.)

\noindent Let $a\in\mathbf{R}^{d},$ $b=\nabla\eta\left(a\right)\in U$,
$1\leq n\leq l$, we write the order $n$ Taylor's expansion of $\nabla\eta$
at $a$ and $F$ at $b$ as follows:
\begin{align*}
\nabla\eta\left(a+x\right) & =\nabla\eta\left(a\right)+\sum_{i=1}^{n}\varphi_{i}\left(x\right)+r\left(x\right)
\end{align*}
where $\varphi_{i}\left(x\right)=\frac{1}{i!}\sum_{\left|\gamma\right|=i}D^{\gamma}\nabla\eta\left(a\right)x^{\gamma},\left|r\left(x\right)\right|=o\left(\left|x\right|^{n}\right),x\in\mathbf{R}^{d},$
and

\noindent
\begin{align*}
F\left(b+y\right) & =F\left(b\right)+\sum_{j=1}^{n}\psi_{j}\left(y\right)+s\left(y\right)
\end{align*}

\noindent where $\psi_{j}\left(y\right)=\psi_{j}\left(y,y,...,y\right)=\frac{1}{j!}F^{\left(j\right)}\left(b\right)\cdot\left(y,y,...,y\right),$
$s\left(y\right)=o\left(\left|y\right|^{n}\right),y\in M_{d\times d}$. 

\noindent We denote $\tilde{\psi}_{j}$ the multi-linear symmetrical
mapping associated with $\psi_{j}$ that is $\tilde{\psi}_{j}\left(y_{1},...,y_{j}\right)=\frac{1}{j!}\sum_{\sigma}\psi_{j}\left(y_{\sigma\left(1\right)},...,y_{\sigma\left(j\right)}\right)$
where the summation is taken over all possible permutations. Next,
by the method outlined in \citep[Section 7.5]{Ca}, the homogeneous
component of order $n$ in the finite expansion of $h=F\circ\nabla\eta$
at $a$ is given by

\noindent
\[
\sum_{j=1}^{n}\sum_{i_{1}+i_{2}+...+i_{j}=n}\tilde{\psi}_{j}\left(\varphi_{i_{1}}\left(x\right),\varphi_{i_{2}}\left(x\right),...,\varphi_{i_{j}}\left(x\right)\right)
\]
\begin{align*}
 & =\sum_{j=1}^{n}\sum_{i_{1}+i_{2}+...+i_{j}=n}\frac{1}{j!}\sum_{\sigma}\psi_{j}\left(\varphi_{i_{\sigma\left(1\right)}}\left(x\right),\varphi_{i_{\sigma\left(2\right)}}\left(x\right),...,\varphi_{i_{\sigma\left(j\right)}}\left(x\right)\right)\\
 & =\sum_{j=1}^{n}\frac{1}{\left(j!\right)^{2}}\sum_{i_{1}+i_{2}+...+i_{j}=n}\sum_{\sigma}F^{\left(j\right)}\left(b\right)\cdot\left(\varphi_{i_{\sigma\left(1\right)}}\left(x\right),\varphi_{i_{\sigma\left(2\right)}}\left(x\right),...,\varphi_{i_{\sigma\left(j\right)}}\left(x\right)\right)
\end{align*}

\noindent and it is equal to $\frac{1}{n!}h^{\left(n\right)}\left(a\right)\cdot\left(x,...,x\right)=\frac{1}{n!}\sum_{\left|\mu\right|=n}D^{\mu}h\left(a\right)x^{\mu}$.
Due to uniqueness of the coefficient of $x^{\mu}$ for each $\mu\in\mathbf{N}_{0}^{d}$
such that $\left|\mu\right|=n$ and the explicit form of $F^{\left(j\right)}$
as multiplication of matrices, $D^{\mu}h\left(a\right)$ must be a
linear combination of terms in the form of 
\[
F^{\left(j\right)}\left(\nabla\eta\left(a\right)\right)\cdot\left(D^{\mu_{1}}\nabla\eta\left(a\right),...,D^{\mu_{j}}\nabla\eta\left(a\right)\right)
\]
where $j=1,...,n$ and $\mu_{1}+\mu_{2}+...+\mu_{j}=\mu.$

\noindent Let $\gamma\in\mathbf{N}_{0}^{d}$ such that $1\leq\left|\gamma\right|=k\leq l$,
then $w^{k}D^{\gamma}\nabla\kappa=w^{k}D^{\gamma}\left(h\left(\kappa\right)\right)$
is a linear combination of terms in the form of 
\begin{align*}
 & F^{\left(j\right)}\left(\nabla\eta\left(\kappa\right)\right)\cdot\left(w^{\left|\mu_{1}\right|}D^{\mu_{1}}\nabla\eta\left(\kappa\right),...,w^{\left|\mu_{j}\right|}D^{\mu_{j}}\nabla\eta\left(\kappa\right)\right)\prod_{i=1}^{n}w^{\left|\alpha_{i}\right|-1}D^{\alpha_{i}}\kappa^{a_{i}}
\end{align*}

\noindent where $\alpha_{1}+\ldots+\alpha_{n}=\gamma,\left|\alpha_{i}\right|\geq1\text{\text{ for }}i=1,...,n,\left\vert \mu\right\vert =n,1\leq n\leq k\leq l$,
and $1\leq a_{i}\leq d$ are component indices. 

\noindent We note that by (i), 
\[
\left|\left(\nabla\eta\left(\kappa\right)\right)^{-1}\right|_{\infty}=\left|\nabla\kappa\right|_{\infty}\leq2.
\]

\noindent If $k=1,$ then by Lemma \ref{lem:weightC}, 
\begin{align*}
\left|wD^{\gamma}\nabla\kappa\right|_{p} & \leq C\left|wD^{2}\eta\left(\kappa\right)\right|_{p}\left|\nabla\kappa\right|_{\infty}\\
 & \leq C\lambda_{t}.
\end{align*}

\noindent Now for $2\leq k\leq l,$ we proceed with a strong induction-
assuming that the first estimate in (iii) holds up to $1\leq k-1\leq l-1.$
We estimate each term in the summation denoting
\[
\mathcal{A}=F^{\left(j\right)}\left(\nabla\eta\left(\kappa\right)\right)\cdot\left(w^{\left|\mu_{1}\right|}D^{\mu_{1}}\nabla\eta\left(\kappa\right),...,w^{\left|\mu_{j}\right|}D^{\mu_{j}}\nabla\eta\left(\kappa\right)\right)\prod_{i=1}^{n}w^{\left|\alpha_{i}\right|-1}D^{\alpha_{i}}\kappa^{a_{i}}.
\]

\noindent If $n<l$ then $l-\left|\mu_{i}\right|\geq l-n>\frac{d}{p}$.
By Lemma \ref{lem:weightC} and Sobolev embedding theorem, 
\begin{align*}
\left|F^{\left(j\right)}\left(\nabla\eta\left(\kappa\right)\right)\cdot\left(w^{\left|\mu_{1}\right|}D^{\mu_{1}}\nabla\eta\left(\kappa\right),...,w^{\left|\mu_{j}\right|}D^{\mu_{j}}\nabla\eta\left(\kappa\right)\right)\right|_{\infty} & \leq C\lambda_{t}^{n}.
\end{align*}

\noindent Therefore$,$ by Corollary \ref{cor-product} with $k>\frac{d}{p}$
and the induction hypothesis, 
\begin{align*}
\mathcal{\left|A\right|}_{p} & \leq C\lambda_{t}^{n}\prod_{i=1,\left|\alpha_{i}\right|\geq2}^{n}\left|w^{\left|\alpha_{i}\right|-1}D^{\alpha_{i}}\kappa^{a_{i}}\right|_{H_{p}^{k-\left|\alpha_{i}\right|}}\\
 & \leq C\lambda_{t}^{n}\prod_{i=1,\left|\alpha_{i}\right|\geq2}^{n}\left|D^{\alpha_{i}}\kappa^{a_{i}}\right|_{H_{d/p+k-1,p}^{k-\left|\alpha_{i}\right|}}\leq C\lambda_{t}^{n}\left|D^{2}\kappa^{a_{i}}\right|_{H_{d/p+k-1,p}^{k-2}}^{n}\\
 & \leq C\lambda_{t}^{k}\lambda_{t}^{kp_{k-1}}=C\lambda_{t}^{p_{k}}.
\end{align*}

\noindent If $n=l$ then $\left|\alpha_{i}\right|=1,\,i=1,...,l$
and thus $\prod_{i=1}^{n}\left|w^{\left|\alpha_{i}\right|-1}D^{\alpha_{i}}\kappa^{a_{i}}\right|_{\infty}\leq C.$ 

\noindent Hence, due to Lemma \ref{lem:weightC} and Corollary \ref{cor-product}
,

\noindent
\begin{align*}
\left|\mathcal{A}\right|_{p}\leq & C\left|F^{\left(j\right)}\left(\nabla\eta\left(\kappa\right)\right)\cdot\left(w^{\left|\mu_{1}\right|}D^{\mu_{1}}\nabla\eta\left(\kappa\right),...,w^{\left|\mu_{j}\right|}D^{\mu_{j}}\nabla\eta\left(\kappa\right)\right)\right|_{p}\\
\leq & C\lambda_{t}^{k}\left|\prod_{i=1}^{j}w^{\left|\mu_{i}\right|}D^{\mu_{i}}\nabla\eta\right|_{p}\leq C\lambda_{t}^{p_{k}}.
\end{align*}

\noindent Therefore, the first estimate in (iii) holds for $k$ completing
the induction. The second estimate follows from multiplying $\mathcal{A}$
by $w^{\sigma}$, replacing $w^{\left|\mu_{1}\right|}D^{\mu_{1}}\nabla\eta\left(\kappa\right)$
with $w^{\sigma+\left|\mu_{1}\right|}D^{\mu_{1}}\nabla\eta\left(\kappa\right),$
and applying the first estimate. 
\end{proof}
Next, we derive a growth estimate for $K_{\eta,\mathbf{h}}$ given
by (\ref{eq:SimplifiedKform}) with a general $\mathbf{h}$ in place
of $\mathbf{h}_{\eta}.$
\begin{lem}
\label{lem:K_growth} Let $p>d,\,l>\frac{d}{p},\,\theta\in\left(1,d\right).$
Then there exist $C,r>0$ so that for all $\zeta\in\mathcal{A}_{M,T}^{p,l,\theta}$
and $\mathbf{h}:\Omega\times\left[0,T\right]\rightarrow H_{\theta+l,p}^{l}\left(\mathbf{R}^{d};\mathbf{R}^{d}\times\mathbf{R}^{d}\right),$ 

\begin{align*}
\left|K_{\eta,\mathbf{h}}\left(t\right)\right|_{H_{\theta+l,p}^{l+1}} & \leq C\lambda_{t}^{r}\left|\mathbf{h}\left(t\right)\right|_{H_{\theta+l,p}^{l}},\,\left(\omega,t\right)\in\Omega\times\left[0,T\right].
\end{align*}
\end{lem}
\begin{proof}
\noindent The variable $t$ is omitted. Now by using Lemma \ref{lem:mik1_smoothing_decomp},
Corollary \ref{cor3:product_w}, Lemma \ref{lem:compose} and Lemma
\ref{lem:kappaGrowth} (i, ii, iii) in order, 
\begin{eqnarray*}
\left|K_{\eta,\mathbf{h}}\right|_{H_{\theta+l,p}^{l+1}} & \leq & C\left\vert \phi_{\eta,\mathbf{h}}\right\vert _{H_{\theta+l,p}^{l}}\\
 & \leq & C\left(1+\left|\nabla\mathbf{\kappa}-I\right|_{H_{\theta+l,p}^{l}}\right)^{2}\left|\mathbf{h}\left(\kappa\right)\right|_{H_{\theta+l,p}^{l}}\\
 & \leq & C\lambda_{t}^{r}\left|\mathbf{h}\right|_{H_{\theta+l,p}^{l}}.
\end{eqnarray*}

\noindent The proof is complete.
\end{proof}

\subsection{Difference Estimates}

In this section, we estimate the difference of spatial inverses. We
write for $i=1,2,$ $\eta_{i}\left(t\right)=\zeta_{i}\left(t\right)+e+\epsilon B_{t}$
and denote their spatial inverses by $\kappa_{i}\left(t\right)$$.$
The exponent of $\lambda_{t}$ will be generically denoted by $r$
and allowed to grow as needed. Due to its simplicity, Lemma \ref{lem:weightC}
will from now be applied without further reference. 
\begin{lem}
\label{Lemma:kappaDiff0} Let $p>d,\,l>\frac{d}{p},\,\theta\in\left(1,d\right).$
Then there exist $C,r>0$ such that for all $\zeta_{1},\zeta_{2}\in\mathcal{A}_{M,T}^{p,l,\theta},$
\[
\left|\kappa_{1}\left(t\right)-\kappa_{2}\left(t\right)\right|_{H_{\theta+l-1,p}^{l}}\leq C\lambda_{t}^{r}\left|\zeta_{1}\left(t\right)-\zeta_{2}\left(t\right)\right|_{H_{\theta+l-1,p}^{l}},\,\left(\omega,t\right)\in\Omega\times\left[0,T\right].
\]

\noindent{}
\end{lem}
\begin{proof}
\noindent The variable $t$ is omitted. We observe the following identity,

\noindent
\begin{align*}
\kappa_{1}-\kappa_{2} & =\left(\kappa_{1}\circ\eta_{2}-\kappa_{1}\circ\eta_{1}\right)\circ\kappa_{2}\\
 & =\left[\int_{0}^{1}\nabla\kappa_{1}\left(a\eta_{2}+\left(1-a\right)\eta_{1}\right)\left(\eta_{2}-\eta_{1}\right)da\right]\circ\kappa_{2}.
\end{align*}

\noindent By changing the variable of integration and Lemma \ref{lem:kappaGrowth}
(i), 
\begin{align*}
 & \left|w^{\sigma-1}\left(\kappa_{1}-\kappa_{2}\right)\right|_{p}\\
\leq & C\lambda_{t}^{\left|\sigma-1\right|}\left|\nabla\kappa_{1}\right|_{\infty}\left|w^{\sigma-1}\left(\eta_{2}-\eta_{1}\right)\right|_{p}\leq C\lambda_{t}^{\left|\sigma-1\right|}\left|w^{\sigma-1}\left(\eta_{2}-\eta_{1}\right)\right|_{p}.
\end{align*}

\noindent Now, we easily check the condition in Corollary \ref{cor:compose-neg}.
Owing to Lemma \ref{lem:kappaGrowth} (i, iii), 
\[
\left|\det\nabla\eta_{2}\right|_{\infty}+\left\Vert \eta_{2}\right\Vert +\left\Vert \kappa_{2}\right\Vert +\left|\nabla\kappa_{2}\right|_{\infty}+\chi_{l\geq2}\left|D^{2}\kappa_{2}\right|_{H_{\theta+l-1,p}^{l-2}}\leq C\lambda_{t}^{r}.
\]

\noindent Therefore, applying Corollary \ref{cor:compose-neg} with
$\delta=\theta-1$ and $\delta^{\prime}=\theta,$
\begin{align*}
\left|\kappa_{1}-\kappa_{2}\right|_{H_{\theta+l-1,p}^{l}} & \leq C\lambda_{t}^{r}\int_{0}^{1}\left|\nabla\kappa_{1}\left(a\eta_{2}+\left(1-a\right)\eta_{1}\right)\left(\eta_{2}-\eta_{1}\right)\right|_{H_{\theta+l-1,p}^{l}}da.\\
 & \leq C\lambda_{t}^{r}\int_{0}^{1}\left|\left(\nabla\kappa_{1}\left(a\eta_{2}+\left(1-a\right)\eta_{1}\right)-I\right)\left(\eta_{2}-\eta_{1}\right)\right|_{H_{\theta+l-1,p}^{l}}da\\
 & +C\lambda_{t}^{r}\left|\eta_{2}-\eta_{1}\right|_{H_{\theta+l-1,p}^{l}}\\
 & =\mathcal{I}_{1}+\mathcal{I}_{2}.
\end{align*}

\noindent Next, we apply Corollary \ref{cor3:product_w} with $l>\frac{d}{p},$
$\delta=\theta-1,\,\delta_{1}=\theta,\,\delta_{2}=\theta-1,$

\noindent
\begin{align*}
\mathcal{I}_{1}\leq & C\lambda_{t}^{r}\int_{0}^{1}\left|\nabla\kappa_{1}\left(a\eta_{2}+\left(1-a\right)\eta_{1}\right)-I\right|_{H_{\theta+l,p}^{l}}\left|\eta_{2}-\eta_{1}\right|_{H_{\theta+l-1,p}^{l}}da.
\end{align*}

\noindent Recalling (\ref{eq:InterInvD}), we have $\left\Vert \left(a\eta_{2}+\left(1-a\right)\eta_{1}\right)^{-1}\right\Vert \leq C\lambda_{t},\,a\in\left[0,1\right].$
In addition, from (\ref{eq:det_Inter}), we have $\left|\det\nabla\left(a\eta_{2}+\left(1-a\right)\eta_{1}\right)^{-1}\right|_{\infty}\leq C.$
Thus, applying Lemma \ref{lem:compose} with $\delta=\delta^{\prime}=\theta$$,$
\[
\mathcal{I}_{1}\leq C\lambda_{t}^{r}\left|\nabla\kappa_{1}-I\right|_{H_{\theta+l,p}^{l}}\left|\eta_{2}-\eta_{1}\right|_{H_{\theta+l-1,p}^{l}}.
\]

\noindent The statement now follows from Lemma \ref{lem:kappaGrowth}
(ii, iii). 
\end{proof}
\begin{rem}
\label{rem:inter_compose} The estimates $\left\Vert \left(a\eta_{2}+\left(1-a\right)\eta_{1}\right)^{-1}\right\Vert \leq C\lambda_{t}$
and $\left|\det\nabla\left(a\eta_{2}+\left(1-a\right)\eta_{1}\right)^{-1}\right|_{\infty}\leq C$
will be needed later.
\end{rem}
We now provide the Lipschitz continuity of $K_{\eta,\mathbf{h}}$
with respect to $\eta$ defined by (\ref{eq:SimplifiedKform}) with
a general $\mathbf{h}$ in place of $\mathbf{h}_{\eta}$. 
\begin{lem}
\label{lem:LipEta} Let $p>d,\,l>1+\frac{d}{p},\,\theta\in\left(1,d\right).$
Then there exist $C,r>0$ such that for all $\zeta_{1},\zeta_{2}\in\mathcal{A}_{M,T}^{p,l,\theta}$
and $\mathbf{h}:\Omega\times\left[0,T\right]\rightarrow H_{\theta+l,p}^{l}\left(\mathbf{R}^{d};\mathbf{R}^{d}\times\mathbf{R}^{d}\right),$
\begin{align*}
 & \left|K_{\eta_{1},\mathbf{h}}\left(t\right)-K_{\eta_{2},\mathbf{h}}\left(t\right)\right|_{H_{\theta+l-1,p}^{l}}\\
 & \leq C\lambda_{t}^{r}\left|\mathbf{h}\left(t\right)\right|_{H_{\theta+l,p}^{l}}\left|\zeta_{1}\left(t\right)-\zeta_{2}\left(t\right)\right|_{H_{\theta+l-1,p}^{l}},\,\left(\omega,t\right)\in\Omega\times\left[0,T\right].
\end{align*}
\end{lem}
\begin{proof}
\noindent The variable $t$ is omitted. Applying Lemma \ref{lem:mik1_smoothing_decomp},

\noindent
\begin{align*}
\left|K_{\eta_{1},\mathbf{h}}-K_{\eta_{2},\mathbf{h}}\right|_{H_{\theta+l-1,p}^{l}} & \leq C\left|\phi_{\eta_{1},\mathbf{h}}-\phi_{\eta_{2},\mathbf{h}}\right|_{H_{\theta+l-1,p}^{l-1}}\\
 & \leq C\left|\left(\nabla\kappa_{1}\right)^{\ast}\mathbf{h}\left(\kappa_{1}\right)\left(\nabla\kappa_{1}-\nabla\kappa_{2}\right)\right|_{H_{\theta+l-1,p}^{l-1}}\\
 & +C\left|\left(\nabla\kappa_{1}\right)^{\ast}\left(\mathbf{h}\left(\kappa_{1}\right)-\mathbf{h}\left(\kappa_{2}\right)\right)\nabla\kappa_{2}\right|_{H_{\theta+l-1,p}^{l-1}}\\
 & +C\left|\left(\left(\nabla\kappa_{1}\right)^{\ast}-\left(\nabla\kappa_{2}\right)^{\ast}\right)\mathbf{h}\left(\kappa_{2}\right)\nabla\kappa_{2}\right|_{H_{\theta+l-1,p}^{l-1}}\\
 & =\mathcal{I}_{1}+\mathcal{I}_{2}+\mathcal{I}_{3}.
\end{align*}

\noindent We estimate $\mathcal{I}_{1}$ and similarly $\mathcal{I}_{3}.$
Applying Corollary \ref{cor3:product_w} with $l-1>\frac{d}{p},$
Lemma \ref{lem:compose} with $l-1\geq0$ and $\delta=\delta^{\prime}=\theta$,
and finally Lemma \ref{lem:kappaGrowth} (i, ii, iii), we derive

\noindent
\begin{align*}
\mathcal{I}_{1} & \leq C\left(\left|\left(\nabla\kappa_{1}\right)^{\ast}-I\right|_{H_{\theta+l-1,p}^{l-1}}+1\right)\left|\mathbf{h}\left(\kappa_{1}\right)\right|_{H_{\theta+l-1,p}^{l-1}}\left|\nabla\kappa_{1}-\nabla\kappa_{2}\right|_{H_{\theta+l-1,p}^{l-1}}\\
 & \leq C\lambda_{t}^{r}\left|\mathbf{h}\right|_{H_{\theta+l-1,p}^{l-1}}\left|\nabla\kappa_{1}-\nabla\kappa_{2}\right|_{H_{\theta+l-1,p}^{l-1}}.
\end{align*}

\noindent We now proceed to estimate $\mathcal{I}_{2}.$ Applying
Corollary \ref{cor3:product_w} with $l-1>\frac{d}{p}$ followed by
Lemma \ref{lem:kappaGrowth} (ii, iii),

\noindent
\begin{align*}
\mathcal{I}_{2} & \leq C\left(1+\left|I-\left(\nabla\kappa_{1}\right)^{\ast}\right|_{H_{\theta+l-1,p}^{l-1}}\right)\left(1+\left|I-\nabla\kappa_{2}\right|_{H_{\theta+l-1,p}^{l-1}}\right)\\
 & \times\int_{0}^{1}\left|\nabla\mathbf{h}\left(a\kappa_{1}+\left(1-a\right)\kappa_{2}\right)\left(\kappa_{1}-\kappa_{2}\right)\right|_{H_{\theta+l-1,p}^{l-1}}da\\
 & \leq C\lambda_{t}^{r}\int_{0}^{1}\left|\nabla\mathbf{h}\left(a\kappa_{1}+\left(1-a\right)\kappa_{2}\right)\left(\kappa_{1}-\kappa_{2}\right)\right|_{H_{\theta+l-1,p}^{l-1}}da.
\end{align*}

\noindent Applying Corollary \ref{cor3:product_w} with $l-1>\frac{d}{p},\,\delta=\theta,\,\delta_{1}=\theta+1,\,\delta_{2}=\theta-1,$

\noindent
\[
\mathcal{\mathcal{I}}_{2}\leq C\lambda_{t}^{r}\int_{0}^{1}\left|\nabla\mathbf{h}\left(a\kappa_{1}+\left(1-a\right)\kappa_{2}\right)\right|_{H_{\theta+l,p}^{l-1}}\left|\kappa_{1}-\kappa_{2}\right|_{H_{\theta+l-2,p}^{l-1}}.
\]
We observe that (\ref{eq:inverseI}) ensures that $a\nabla\kappa_{1}+\text{\ensuremath{\left(1-a\right)}}\nabla\kappa_{2}$
is close to the identity matrix. Indeed, 
\begin{align*}
\left|a\nabla\kappa_{1}+\text{\ensuremath{\left(1-a\right)}}\nabla\kappa_{2}-I\right|_{\infty} & \leq a\left|\nabla\kappa_{1}-I\right|_{\infty}+\left(1-a\right)\left|\nabla\kappa_{2}-I\right|_{\infty}\leq\frac{1}{2d}.
\end{align*}

\noindent Due to Ostrowski's lower bound for determinants, $\det\nabla\left(a\kappa_{1}+\left(1-a\right)\kappa_{2}\right)\geq c>0$.
By the inverse function theorem, $a\kappa_{1}+\text{\ensuremath{\left(1-a\right)}}\kappa_{2}$
has the spatial inverse denoted by $\left(a\kappa_{1}+\text{\ensuremath{\left(1-a\right)}}\kappa_{2}\right)^{-1}$.
By the same calculation as (\ref{eq:eta_inv_db}) and Lemma \ref{lem:kappaGrowth}
(i),

\noindent
\begin{align*}
 & \left\Vert \left(a\kappa_{1}+\left(1-a\right)\kappa_{2}\right)^{-1}\right\Vert \\
 & \leq\left(1+\left(a\kappa_{1}\left(0\right)+\left(1-a\right)\kappa_{2}\left(0\right)\right)\right)\left|\nabla\left(a\kappa_{1}+\left(1-a\right)\kappa_{2}\right)^{-1}\right|_{\infty}\\
 & \leq C\lambda_{t},\,a\in\left[0,1\right].
\end{align*}

\noindent Finally, applying Lemma \ref{lem:compose} with $l-1\geq0,\,\delta=\theta+1,\,\delta^{\prime}=\theta,$
and Lemma \ref{lem:kappaGrowth} (i, iii)
\[
\mathcal{\mathcal{I}}_{2}\leq C\lambda_{t}^{r}\left|\nabla\mathbf{h}\right|_{H_{\theta+l,p}^{l-1}}\left|\kappa_{1}-\kappa_{2}\right|_{H_{\theta+l-2,p}^{l-1}}.
\]

\noindent The proof is completed by Lemma \ref{Lemma:kappaDiff0}.
\end{proof}
The following Lipschitz continuity will also be needed.
\begin{lem}
\label{lem:diff_h} Let $p>d,\,l>1+\frac{d}{p},\,\theta\in\left(1,d\right),\,N>0$
and $\sup_{\left(\omega,t\right)\in\Omega\times\left[0,T\right]}\left|G\left(t\right)\right|_{H_{\theta+l}^{l+1}}\leq N.$
Denote for $i=1,2,$ 
\begin{align*}
\mathbf{h}_{\eta_{i}}\left(t\right)= & \int_{0}^{t}\left(\nabla\eta_{i}\left(s\right)\right)^{\ast}\nabla G\left(s,\eta_{i}\left(s\right)\right)\left(\nabla\eta_{i}\left(s\right)\right)ds,\,\left(\omega,t\right)\in\Omega\times\left[0,T\right].
\end{align*}

\noindent Then there exist $C,r>0$ such that for all $\zeta_{1},\zeta_{2}\in\mathcal{A}_{M,T}^{p,l,\theta},$
\begin{align*}
\left|\mathbf{h}_{\eta_{1}}\left(t\right)-\mathbf{h}_{\eta_{2}}\left(t\right)\right|_{H_{\theta+l-1}^{l-1}} & \leq Ct\lambda_{t}^{r}\left|\zeta_{1}-\zeta_{2}\right|_{\mathbb{\sum}_{\theta+l-1,p}^{l}\left(t\right)},\,\left(\omega,t\right)\in\Omega\times\left[0,T\right].
\end{align*}
\end{lem}
\begin{proof}
\noindent We split the difference as follows:

\noindent
\begin{align*}
 & \left|\mathbf{h}_{\eta_{1}}\left(t\right)-\mathbf{h}_{\eta_{2}}\left(t\right)\right|_{H_{\theta+l-1}^{l-1}}\\
 & \leq\left|\int_{0}^{t}\left(\nabla\eta_{1}\left(s\right)-\nabla\eta_{2}\left(s\right)\right)^{\ast}\nabla G\left(s,\eta_{1}\left(s\right)\right)\left(\nabla\eta_{1}\left(s\right)\right)ds\right|_{H_{\theta+l-1}^{l-1}}\\
 & +\left|\int_{0}^{t}\left(\nabla\eta_{2}\left(s\right)\right)^{\ast}\left(\nabla G\left(s,\eta_{1}\left(s\right)\right)-\nabla G\left(s,\eta_{2}\left(s\right)\right)\right)\left(\nabla\eta_{1}\left(s\right)\right)ds\right|_{H_{\theta+l-1}^{l-1}}\\
 & +\left|\int_{0}^{t}\left(\nabla\eta_{2}\left(s\right)\right)^{\ast}\nabla G\left(s,\eta_{2}\left(s\right)\right)\left(\nabla\eta_{1}\left(s\right)-\nabla\eta_{2}\left(s\right)\right)ds\right|_{H_{\theta+l-1}^{l-1}}\\
 & =\mathcal{I}_{1}+\mathcal{I}_{2}+\mathcal{I}_{3}.
\end{align*}

\noindent We start with $\mathcal{I}_{2}.$ Applying Corollary \ref{cor3:product_w}
with $l-1>\frac{d}{p},$
\begin{align*}
\mathcal{I}_{2} & \leq C\prod_{i=1,2}\left(1+\left|I-\nabla\eta_{i}\right|_{\mathbb{\sum}_{\theta+l-1,p}^{l-1}\left(t\right)}\right)\left|\int_{0}^{t}\left(\nabla G\left(s,\eta_{1}\left(s\right)\right)-\nabla G\left(s,\eta_{2}\left(s\right)\right)\right)ds\right|_{H_{\theta+l-1,p}^{l-1}}\\
 & \leq C\left|\int_{0}^{t}\left(\nabla G\left(s,\eta_{1}\left(s\right)\right)-\nabla G\left(s,\eta_{2}\left(s\right)\right)\right)ds\right|_{H_{\theta+l-1,p}^{l-1}}.
\end{align*}

\noindent By the fundamental theorem of calculus followed by Corollary
\ref{cor3:product_w} with $l-1>\frac{d}{p},$ $\delta=\theta,$ $\delta_{1}=\theta+1,$
$\delta_{2}=\theta-1,$
\[
\mathcal{I}_{2}\leq C\int_{0}^{t}\left|D^{2}G\left(s,a\eta_{1}\left(s\right)+\left(1-a\right)\eta_{2}\right)\right|_{H_{\theta+l,p}^{l-1}}ds\left|\zeta_{1}-\zeta_{2}\right|_{\mathbb{\sum}_{\theta+l-2,p}^{l-1}\left(t\right)}.
\]

\noindent Finally, using Lemma \ref{lem:compose} with $l-1\geq0$,
$\delta=\theta+1,$ $\delta^{\prime}=\theta$, and Remark \ref{rem:inter_compose},
\[
\mathcal{I}_{2}\leq Ct\lambda_{t}^{r}\left|D^{2}G\right|_{\mathbb{\sum}_{\theta+l,p}^{l-1}\left(t\right)}\left|\zeta_{1}-\zeta_{2}\right|_{\mathbb{\sum}_{\theta+l-2,p}^{l-1}\left(t\right)}.
\]

\noindent Now, we now estimate $\mathcal{I}_{1}$ and similarly $\mathcal{I}_{3}$.
Applying Corollary \ref{cor3:product_w} with $l-1>\frac{d}{p},$
\[
\mathcal{I}_{1}\leq C\left|\nabla\zeta_{1}-\nabla\zeta_{2}\right|_{\mathbb{\sum}_{\theta+l-1,p}^{l-1}\left(t\right)}\int_{0}^{t}\left|\nabla G\left(s,\eta_{1}\left(s\right)\right)\right|_{H_{\theta+l-1,p}^{l-1}}ds.
\]

\noindent Finally, using Lemma \ref{lem:compose} with $l-1\geq0$,
$\delta=\theta,$ $\delta^{\prime}=\theta,$
\[
\mathcal{I}_{1}\leq Ct\lambda_{t}^{r}\left|\nabla\zeta_{1}-\nabla\zeta_{2}\right|_{\mathbb{\sum}_{\theta+l-1,p}^{l-1}\left(t\right)}\left|\nabla G\right|_{\sum_{\theta+l-1,p}^{l-1}\left(t\right)}.
\]

\noindent The proof is complete.
\end{proof}

\section{Proof of the Main Theorem}

In this section, we construct a solution of the flow equation via
iteration. Specifically, we will show a contraction in an appropriate
function space of the mapping $\eta\rightarrow K_{\eta,\mathbf{h}_{\eta}}$
given by (\ref{eq:SimplifiedKform}). 

\subsection{Proof of Theorem \ref{thm:main_flow}}

We now prove Theorem \ref{thm:main_flow}. 
\begin{proof}
\noindent We will show the existence for all $\omega\in\Omega$ and
all estimates will be independent of $\omega.$ Fixing $0<T<\infty$,
we consider the mapping $\mathcal{L}:\zeta\rightarrow\int_{0}^{t}\mathbf{E}_{s}^{W}K_{\eta,\mathbf{h}_{\eta}}\left(s,\eta\left(s\right)\right)ds=\tilde{\zeta}\left(t\right),\,t\in\left[0,T\right]$
where 
\begin{align*}
K_{\eta,\mathbf{h_{\eta}}}^{j}\left(t,x\right) & =\int\Gamma_{i}\left(x-z\right)\left[\phi_{\eta,\mathbf{h_{\eta}}}^{ji}\left(t,z\right)-\phi_{\eta,\mathbf{h}_{\eta}}^{ij}\left(t,z\right)\right]dz,
\end{align*}

\noindent
\[
\phi_{\eta,\mathbf{h}}\left(t,x\right)=\left(\nabla\kappa\left(t,x\right)\right)^{\ast}\mathbf{h}\left(t,\kappa\left(t,x\right)\right)\nabla\kappa\left(t,x\right),
\]

\noindent and
\[
\mathbf{h}_{\eta}\left(t,x\right)=\nabla\mathbf{u_{0}}\left(x\right)+\int_{0}^{t}\left(\nabla\eta\left(s,x\right)\right)^{\ast}\nabla G\left(s,\eta\left(s,x\right)\right)\left(\nabla\eta\left(s,x\right)\right)ds,
\]

\noindent for all $\left(\omega,t,x\right)\in\Omega\times\left[0,T\right]\times\mathbf{R}^{d}.$

\vspace{10pt}

\noindent Recall that $\zeta\in\mathcal{A}_{M,T}^{p,l,\theta}$ if
for all $\left(\omega,t\right)\in\Omega\times\left[0,T\right],$ $\zeta\left(t\right)$
is a continuously differentiable and 

\noindent (i) $\left\vert \nabla\zeta\left(t\right)\right\vert _{\infty}\leq\frac{1}{2d^{2}}$,
$\left|\zeta\left(t,0\right)\right|\leq M,$

\noindent (ii) $\left|\zeta\left(t\right)\right|_{H_{\theta+l,p}^{l+1}}\leq M.$ 

\vspace{10pt}

\noindent We mention that by Lemma \ref{lem:kappaGrowth} (i), $\left\Vert \kappa\left(s\right)\right\Vert \leq C\lambda_{s}$
which will be used several times. Applying Corollary \ref{cor3:product_w}
and Lemma \ref{lem:compose} with $\delta=\delta^{\prime}=\theta,$ 

\noindent
\begin{align}
\left|\mathbf{h}_{\eta}\left(s\right)\right|_{H_{\theta+l,p}^{l}} & \leq\left|\nabla\mathbf{u}_{0}\right|_{H_{\theta+l,p}^{l}}+Cs\lambda_{s}^{r}\left|\nabla G\right|_{\mathbb{\sum}_{\theta+l,p}^{l}\left(s\right)}\leq C\lambda_{s}^{r}.\label{eq:bound_h}
\end{align}

\noindent Therefore, by Lemma \ref{lem:K_growth}, 
\begin{equation}
\left|K_{\eta,\mathbf{h}_{\eta}}\left(s\right)\right|_{H_{\theta+l,p}^{l+1}}\leq C\lambda_{s}^{r}\left|\mathbf{h}_{\eta}\left(s\right)\right|_{H_{\theta+l,p}^{l}}\leq C\lambda_{s}^{r},\label{eq:uniform_K}
\end{equation}

\noindent and by Corollary \ref{cor:compose-neg} with $l+1\geq0$,
$\delta=\theta-1$, $\delta^{\prime}=\theta,$ 
\begin{align}
\mathbf{E}_{s}^{W}\left|K_{\eta,\mathbf{h}_{\eta}}\left(s,\eta\left(s\right)\right)\right|_{H_{\theta+l,p}^{l+1}} & \leq C\mathbf{E}_{s}^{W}\lambda_{s}^{r}\leq C.\label{eq:EK}
\end{align}

\noindent Hence, 

\noindent
\begin{align*}
\left|\tilde{\zeta}\left(t\right)\right|_{H_{\theta+l,p}^{l+1}} & \leq Ct.
\end{align*}

\noindent By Sobolev embedding theorem, we have $\left|\tilde{\zeta}\left(t,0\right)\right|\leq Ct$,
$\left|\nabla\tilde{\zeta}\left(t\right)\right|_{\infty}\leq Ct$
and $\tilde{\zeta}$ is continuously differentiable. We emphasize
that $C$ depends on $M$ and $T$ but is independent of $\zeta.$
Therefore, fixing $M>0$ and making $T$ smaller i.e., $CT\leq M,$
then $\mathcal{L}:\zeta\rightarrow\int_{0}^{t}\mathbf{E}_{s}^{W}K_{\eta,\mathbf{h}_{\eta}}\left(s,\eta\left(s\right)\right)ds=\tilde{\zeta}\left(t\right)$
maps $\mathcal{A}_{M,T,l}^{p,l,\theta}$ into $\mathcal{A}_{M,T,l}^{p,l,\theta}.$

\noindent Next, we show the contraction in the weaker $\sum_{\theta+l-1,p}^{l}\left(T\right)$
norm. We assume that $\zeta,\tilde{\zeta}\in\mathcal{A}_{M,T}^{p,l,\theta}$
are $\mathbb{F}^{W}-$adapted.
\begin{align*}
 & \int_{0}^{t}\left|\mathbf{E}_{s}^{W}K_{\eta,\mathbf{h}_{\eta}}\left(s,\eta\left(s\right)\right)-\mathbf{E}_{s}^{W}K_{\tilde{\eta},\mathbf{h}_{\tilde{\eta}}}\left(s,\tilde{\eta}\left(s\right)\right)\right|_{H_{\theta+l-1,p}^{l}}ds\\
\leq & \int_{0}^{t}\mathbf{E}_{s}^{W}\left|K_{\eta,\mathbf{h}_{\eta}}\left(s,\eta\left(s\right)\right)-K_{\eta,\mathbf{h}_{\eta}}\left(s,\tilde{\eta}\left(s\right)\right)\right|_{H_{\theta+l-1,p}^{l}}ds\\
+ & \int_{0}^{t}\mathbf{E}_{s}^{W}\left|K_{\eta,\mathbf{h}_{\eta}}\left(s,\tilde{\eta}\left(s\right)\right)-K_{\tilde{\eta},\mathbf{h}_{\tilde{\eta}}}\left(s,\tilde{\eta}\left(s\right)\right)\right|_{H_{\theta+l-1,p}^{l}}ds\\
= & \int_{0}^{t}\mathcal{I}_{1}\left(s\right)ds+\int_{0}^{t}\mathcal{I}_{2}\left(s\right)ds.
\end{align*}

\noindent Estimate of $\mathcal{I}_{1}$: Applying Corollary \ref{cor3:product_w}
with $\delta=\theta-1$, $\delta_{1}=\theta,$ $\delta_{2}=\theta-1,$ 

\noindent
\begin{align*}
\mathcal{I}_{1}\left(s\right) & =\mathbf{E}_{s}^{W}\left|K_{\eta,\mathbf{h}_{\eta}}\left(s,\eta\left(s\right)\right)-K_{\eta,\mathbf{h}_{\eta}}\left(s,\tilde{\eta}\left(s\right)\right)\right|_{H_{\theta+l-1,p}^{l}}\\
\leq & \mathbf{E}_{s}^{W}\left|\int_{0}^{1}\nabla K_{\eta,\mathbf{h}_{\eta}}\left(s,a\eta\left(s\right)+\left(1-a\right)\tilde{\eta}\left(s\right)\right)\left(\eta\left(s\right)-\tilde{\eta}\left(s\right)\right)da\right|_{H_{\theta+l-1,p}^{l}}\\
\leq & C\int_{0}^{1}\mathbf{E}_{s}^{W}\left|\nabla K_{\eta,\mathbf{h}_{\eta}}\left(s,a\eta\left(s\right)+\left(1-a\right)\tilde{\eta}\left(s\right)\right)\right|_{H_{\theta+l,p}^{l}}da\ensuremath{\left|\eta\left(s\right)-\tilde{\eta}\left(s\right)\right|_{H_{\theta+l-1,p}^{l}}}.
\end{align*}

\noindent By Lemma \ref{lem:compose} with $\delta=\delta^{\prime}=\theta$,
(\ref{eq:uniform_K}), and Remark \ref{rem:inter_compose}, 
\[
\mathcal{I}_{1}\left(s\right)\leq C\left|\zeta\text{\ensuremath{\left(s\right)}}-\tilde{\zeta}\left(s\right)\right|_{H_{\theta+l-1,p}^{l}}.
\]
 Estimate of $\mathcal{I}_{2}$: By Corollary \ref{cor:compose-neg}
with $\delta=\theta-1,$ $\delta^{\prime}=\theta$, 

\noindent
\begin{align*}
\mathcal{I}_{2}\left(s\right) & =\mathbf{E}_{s}^{W}\left|\left(K_{\eta,\mathbf{h}_{\eta}}\left(s\right)-K_{\tilde{\eta},\mathbf{h}_{\tilde{\eta}}}\left(s\right)\right)\circ\tilde{\eta}\left(s\right)\right|_{H_{\theta+l-1,p}^{l}}\\
 & \leq\mathbf{E}_{s}^{W}\left|K_{\eta,\mathbf{h}_{\eta}-\text{\ensuremath{\mathbf{h}_{\tilde{\eta}}}}}\left(s\right)\circ\tilde{\eta}\left(s\right)\right|_{H_{\theta+l-1,p}^{l}}+\mathbf{E}_{s}^{W}\left|\left(K_{\eta,\mathbf{h}_{\tilde{\eta}}}\left(s\right)-K_{\tilde{\eta},\mathbf{h}_{\tilde{\eta}}}\left(s\right)\right)\circ\tilde{\eta}\left(s\right)\right|_{H_{\theta+l-1,p}^{l}}\\
 & \leq C\mathbf{E}_{s}^{W}\left[\lambda_{s}^{r}\left|K_{\eta,\mathbf{h}_{\eta}-\text{\ensuremath{\mathbf{h}_{\tilde{\eta}}}}}\left(s\right)\right|_{H_{\theta+l-1,p}^{l}}\right]+C\mathbf{E}_{s}^{W}\left[\lambda_{s}^{r}\left|K_{\eta,\mathbf{h}_{\tilde{\eta}}}\left(s\right)-K_{\tilde{\eta},\mathbf{h}_{\tilde{\eta}}}\left(s\right)\right|_{H_{\theta+l-1,p}^{l}}\right]\\
 & =\mathcal{I}_{21}\left(s\right)+\mathcal{I}_{22}\left(s\right).
\end{align*}

\noindent Applying Lemma \ref{lem:K_growth} with $l-1>\frac{d}{p},$
\begin{align*}
\mathcal{I}_{21}\left(s\right) & \leq C\mathbf{E}_{s}^{W}\left[\lambda_{s}^{r}\left|\mathbf{h}_{\eta}\left(s\right)-\mathbf{h}_{\tilde{\eta}}\left(s\right)\right|_{H_{\theta+l-1,p}^{l-1}}\right].
\end{align*}

\noindent Owing to Lemma \ref{lem:diff_h}, $\mathcal{I}_{21}\left(s\right)\leq C\left|\zeta-\tilde{\zeta}\right|_{\Sigma_{\theta+l-1,p}^{l}\left(s\right)}.$

\noindent By Lemma \ref{lem:LipEta} and (\ref{eq:bound_h}), 
\begin{align*}
\mathcal{I}_{22}\left(s\right) & \leq C\left|\zeta\left(s\right)-\tilde{\zeta}\left(s\right)\right|_{H_{\theta+l-1,p}^{l}}.
\end{align*}

\noindent Combining estimates of $\mathcal{I}_{1}$ and $\mathcal{I}_{2}$,
we derive 
\begin{equation}
\left|\mathcal{L}\left(\zeta\right)-\mathcal{L}\left(\tilde{\zeta}\right)\right|_{\Sigma_{\theta+l-1,p}^{l}\left(T\right)}\leq C\left|\zeta-\tilde{\zeta}\right|_{\Sigma_{\theta+l-1,p}^{l}\left(T\right)}.\label{eq:contraction}
\end{equation}

\noindent By a standard successive iteration starting with $\zeta^{\left(0\right)}\left(t\right)=0,$
\[
\zeta^{\left(n+1\right)}\left(t\right)=\int_{0}^{t}\mathbf{E}_{s}^{W}K_{\eta^{\left(n\right)},\mathbf{h}_{\eta^{\left(n\right)}}}\left(s,\eta^{\left(n\right)}\left(s\right)\right)ds,\,t\in\left[0,T\right],
\]
there exists $\zeta^{\ast}\in\mathbb{\sum}_{\theta+l-1,p}^{l}\left(T\right)$
such that 
\[
\zeta^{\ast}\left(t\right)=\int_{0}^{t}\mathbf{E}_{s}^{W}K_{\eta^{\ast},\mathbf{h}_{\eta^{\ast}}}\left(s,\eta^{\ast}\left(s\right)\right)ds
\]
as an equality in $H_{\theta+l-1,p}^{l}\left(\mathbf{R}^{d};\mathbf{R}^{d}\right).$
Owing to Lemma \ref{lem:weaker_norm}, $\zeta^{\ast}\in\mathbb{\sum}_{\theta+l,p}^{l+1}\left(T\right).$
Finally, we show that $\zeta\in C\left(\left[0,T\right],H_{\theta+l,p}^{l+1}\left(\mathbf{R}^{d};\mathbf{R}^{d}\right)\right).$
Clearly, by passing to the limit, $\zeta^{\ast}\in\mathcal{A}_{M,T}^{p,l,\theta}$
possibly with a larger $M$. Hence, by (\ref{eq:EK}), 
\begin{align*}
\left|\zeta^{\ast}\left(t\right)-\zeta^{\ast}\left(t^{\prime}\right)\right|_{H_{\theta+l,p}^{l+1}} & \leq\int_{t^{\prime}}^{t}\mathbf{E}_{s}^{W}\left|K_{\eta^{\ast},\mathbf{h}_{\eta^{\ast}}}\left(s,\eta^{\ast}\left(s\right)\right)\right|_{H_{\theta+l,p}^{l+1}}ds\\
 & \leq C\left(t-t^{\prime}\right).
\end{align*}

\noindent The proof is complete.
\end{proof}

\subsection{Proof of Theorem \ref{thm:main_velocity}}

We now prove Theorem \ref{thm:main_velocity}. 
\begin{proof}
We start by checking the assumptions of Lemma \ref{lem:simplified_K}.

\noindent (i) By applying Corollary \ref{cor3:product_w} and Lemma
\ref{lem:compose}, 
\[
\left(\nabla\mathbf{\kappa}\left(t\right)\right)^{\ast}\mathbf{g}_{\eta}\left(t,\mathbf{\kappa}\left(t\right)\right)\in H_{\theta+l,p}^{l}\left(\mathbf{R}^{d};\mathbf{R}^{d}\right),\,\left(\omega,t\right)\in\Omega\times\left[0,T\right].
\]

\noindent which also justifies condition (i) of Lemma \ref{lem:flow1}.

\noindent (ii) follows from differentiability of $\zeta\left(t\right)$
and Lemma \ref{lem:kappaGrowth}. 

\noindent (iii) follows from continuity of $\left|\zeta\left(t\right)\right|_{H_{\theta+l,p}^{l+1}}$
and Sobolev embedding theorem. 

\noindent Hence, by (\ref{eq:S2K}) in Lemma \ref{lem:simplified_K},
\[
\mathcal{S}\left(\nabla\mathbf{\kappa}\left(t\right)\right)^{\ast}\mathbf{g}_{\eta}\left(t,\mathbf{\kappa}\left(t\right)\right)=K_{\eta,\mathbf{h}_{\eta}}\left(t\right),\,\left(\omega,t\right)\in\Omega\times\left[0,T\right].
\]
We let
\[
\mathbf{u}\left(t\right)=\mathbf{E}_{t}^{W}\mathcal{S}\left(\nabla\mathbf{\kappa}\left(t\right)\right)^{\ast}\mathbf{g}_{\eta}\left(t,\mathbf{\kappa}\left(t\right)\right)=\mathbf{E}_{t}^{W}K_{\eta,\mathbf{h}_{\eta}}\left(t\right),\,\left(\omega,t\right)\in\Omega\times\left[0,T\right].
\]

\noindent By (\ref{eq:uniform_K}), 
\[
\sup_{\left(\omega,t\right)\in\Omega\times\left[0,T\right]}\left|\mathbf{u}\left(t\right)\right|_{H_{\theta+l,p}^{l+1}}\leq\sup_{\left(\omega,t\right)\in\Omega\times\left[0,T\right]}\mathbf{E}_{t}^{W}\left|K_{\eta,\mathbf{h}_{\eta}}\left(t\right)\right|_{H_{\theta+l,p}^{l+1}}\leq N<\infty.
\]
Because $l+1>3+\frac{d}{p},$ we have by Sobolev embedding theorem
that 
\[
\sup_{\left(\omega,t\right)\in\Omega\times\left[0,T\right]}\left|\mathbf{u}\left(t\right)\right|_{\mathcal{C}^{3+\alpha}}\leq N<\infty
\]
for some $\alpha\in\left(0,1\right].$ Furthermore, the above inequality
also justifies condition (ii) of Lemma \ref{lem:flow1}. Therefore,
applying Lemma \ref{lem:flow1}, $\mathbf{u}\left(t\right)$ solves
(\ref{eq:LinSNSE}) as an equality in $H_{\theta+l-2,p}^{l-2}\left(\mathbf{R}^{d};\mathbf{R}^{d}\right).$ 

\noindent By Corollary \ref{cor3:product_w} with $\delta=\theta,\,\delta_{1}=\theta,\,\delta_{2}=\theta-1,$
\[
\left(\nabla\mathbf{u}\left(t\right)\right)^{\ast}\mathbf{u}\left(t\right)=\frac{1}{2}\nabla\left|\mathbf{u}\left(t\right)\right|^{2}\in H_{\theta+l,p}^{l}\left(\mathbf{R}^{d};\mathbf{R}^{d}\right).
\]
Thus, $\left(\nabla\mathbf{u}\left(t\right)\right)^{\ast}\mathbf{u}\left(t\right)$
disappears under the Solenoidal projection in $H_{\theta+l-2,p}^{l-2}\left(\mathbf{R}^{d};\mathbf{R}^{d}\right).$
Therefore, $\mathbf{u}\left(t\right)$ solves (\ref{eq:NSE-cor})
as an equality in $H_{\theta+l-2,p}^{l-2}\left(\mathbf{R}^{d};\mathbf{R}^{d}\right).$ 
\end{proof}

\section{Appendix}

\subsection{Weaker Norm}

We mention a result on convergence of functions in $H_{\delta,p}^{l}\left(\mathbf{R}^{d};\mathbf{R}^{d}\right)$
which is needed for constructing a solution. 
\begin{lem}
\label{lem:weaker_norm} Let $p>1,\,l\geq0,\,\delta\geq0.$ If $\zeta_{n}\rightarrow\zeta$
in $H_{\delta,p}^{l}\left(\mathbf{R}^{d};\mathbf{R}^{d}\right)$ and
$\left|\zeta_{n}\right|_{H_{\delta+1,p}^{l+1}}\leq M<\infty$ for
all $n\geq1$ then $\zeta\in H_{\delta+1,p}^{l+1}\left(\mathbf{R}^{d};\mathbf{R}^{d}\right).$
\end{lem}
\begin{proof}
\noindent Without loss of generality, we assume that $\zeta_{n},\zeta\in H_{\delta,p}^{l}\left(\mathbf{R}^{d};\mathbf{R}\right)$
instead of $H_{\delta,p}^{l}\left(\mathbf{R}^{d};\mathbf{R}^{d}\right).$
The idea is to show that $w^{\sigma}\partial^{l+1}\zeta\in L_{p}\left(\mathbf{R}^{d}\right)$
where $\sigma=\delta+1-\frac{d}{p}.$ Let $\gamma\in\mathbf{N}_{0}^{d}$
with $\left|\gamma\right|=l$ and $1\leq i\leq d$.

\noindent By a standard argument, we take $\varphi\in C_{0}^{\infty}\left(\mathbf{R}^{d}\right)$
then on some subsequence of $\left\{ \zeta_{n}\right\} $ there exist
$f,g\in L_{p}\left(\mathbf{R}^{d}\right)$ such that

\noindent
\begin{align*}
\int\left(w^{\sigma}\partial_{i}\partial^{\gamma}\zeta_{n}\right)\left(x\right)\varphi\left(x\right)dx & \rightarrow\int f\left(x\right)\varphi\left(x\right)dx\\
\int\left(w^{\sigma-1}\partial_{i}w\partial^{\gamma}\zeta_{n}\right)\left(x\right)\varphi\left(x\right)dx & \rightarrow\int g\left(x\right)\varphi\left(x\right)dx.
\end{align*}

\noindent On the other hand, using integration by parts,

\noindent
\begin{align*}
 & -\int\left(w^{\sigma}\partial_{i}\partial^{\gamma}\zeta_{n}\right)\left(x\right)\varphi\left(x\right)+\left(\sigma w^{\sigma-1}\partial_{i}w\partial^{\gamma}\zeta_{n}\right)\left(x\right)\varphi\left(x\right)dx\\
 & =\int\left(w^{\sigma}\partial^{\gamma}\zeta_{n}\right)\left(x\right)\partial_{i}\varphi\left(x\right)dx.
\end{align*}

\noindent Taking limit $n\rightarrow\infty$, 
\begin{align*}
-\int f\left(x\right)\varphi\left(x\right)dx-\sigma\int g\left(x\right)\varphi\left(x\right)dx & =\int\left(w^{\sigma}\partial^{\gamma}\zeta\right)\left(x\right)\partial_{i}\varphi\left(x\right)dx.
\end{align*}

\noindent Therefore, $w^{\sigma}\partial^{\gamma}\zeta$ has a weak
derivative, namely $f+\sigma g\in L_{p}\left(\mathbf{R}^{d}\right)$.
Now because $w^{\sigma-1}\partial^{\gamma}\zeta\in L_{p}\left(\mathbf{R}^{d}\right)$,
we conclude that $w^{\sigma}\partial_{i}\partial^{\gamma}\zeta\in L_{p}\left(\mathbf{R}^{d}\right)$.
Since $\gamma,i$ are arbitrary, the conclusion follows.
\end{proof}

\subsection{Weight Function}

Let $B_{r}\left(x\right)=\left\{ y\in\mathbf{R}^{d}:\left\vert x-y\right\vert <r\right\} ,$$x\in\mathbf{R}^{d},B_{r}=B_{r}\left(0\right),r>0$.
Let $p>1,\frac{1}{p}+\frac{1}{q}=1$. We say that a non-negative function
$w$ on $\mathbf{R}^{d}$ is of class $A_{p}$ if 
\[
\left(\frac{1}{\left\vert B_{R}\right\vert }\int_{B_{R}}w\left(x+y\right)dy\right)\left(\frac{1}{\left\vert B_{R}\right\vert }\int_{B_{R}}w\left(x+y\right)^{-q/p}dy\right)^{p/q}\leq C,\,x\in\mathbf{R}^{d},\,R>0.
\]

We show an important property of the weight function $w\left(x\right)=\left(1+\left\vert x\right\vert ^{2}\right)^{1/2},x\in\mathbf{R}^{d}.$
\begin{lem}
\label{al1}Let $p>1,\,d\geq1,\,\alpha\in\left(-d,d\left(p-1\right)\right),$
then $w^{\alpha}\in A_{p}$.
\end{lem}
\begin{rem}
$d-\frac{\alpha q}{p}>0$ is equivalent to $\alpha<d\left(p-1\right)$.
Also, for $\theta\in\left(1,d\right)$, $\alpha=\theta-\frac{d}{p}\in\left(-d,d\left(p-1\right)\right).$
\end{rem}
\begin{proof}
For each $\alpha\in\mathbf{R}$, 
\[
\frac{1}{\left\vert B_{R}\right\vert }\int_{B_{R}}w\left(y\right)^{\alpha}dy\leq CR^{-d}\int_{0}^{R}\left(1+r\right)^{\alpha}r^{d}\frac{dr}{r}\leq C\text{ if }R\in\left(0,1\right].
\]

\noindent If $R>1$ then 
\begin{eqnarray*}
\int_{0}^{R}\left(1+r\right)^{\alpha}r^{d}\frac{dr}{r} & = & \int_{0}^{1}\left(1+r\right)^{\alpha}r^{d}\frac{dr}{r}+\int_{1}^{R}\left(1+r\right)^{\alpha}r^{d}\frac{dr}{r}\\
 & \leq & CR^{d+\alpha}\text{ if }d+\alpha>0.
\end{eqnarray*}

\noindent Therefore, for $R>1,$ noting $d+\alpha>0,$ 
\begin{equation}
\frac{1}{\left\vert B_{R}\right\vert }\int_{B_{R}}w\left(y\right)^{\alpha}dy\leq CR^{\alpha}.\label{fa2}
\end{equation}

\noindent Hence, for $R>1,$ noting $d+\alpha>0$ and $d-\frac{\alpha q}{p}>0,$

\begin{equation}
\left(\frac{1}{\left\vert B_{R}\right\vert }\int_{B_{R}}w\left(y\right)^{\alpha}dy\right)\left(\frac{1}{\left\vert B_{R}\right\vert }\int_{B_{R}}w\left(y\right)^{-\alpha q/p}dy\right)^{p/q}\leq C.\label{eq:onlyy_1}
\end{equation}
Consider now
\[
\frac{1}{\left\vert B_{R}\left(x\right)\right\vert }\int_{B_{R}\left(x\right)}w\left(y\right)^{\alpha}dy=\frac{1}{\left\vert B_{R}\right\vert }\int_{B_{R}}w\left(x+y\right)^{\alpha}dy,\,x\in\mathbf{R}^{d}.
\]
Since
\[
\frac{1}{2}w\left(x\right)w\left(y\right)^{-1}\leq w\left(x+y\right)\leq2w\left(x\right)w\left(y\right),x,y\in\mathbf{R}^{d},
\]
it follows that
\[
\frac{1}{\left\vert B_{R}\right\vert }\int_{B_{R}}w\left(x+y\right)^{\alpha}dy\leq2^{\alpha}w\left(x\right)^{\alpha}\frac{1}{\left\vert B_{R}\right\vert }\int_{B_{R}}w\left(y\right)^{\alpha}dy\text{ if }\alpha\geq0,
\]
and
\[
\frac{1}{\left\vert B_{R}\right\vert }\int_{B_{R}}w\left(x+y\right)^{\alpha}dy\leq2^{-\alpha}w\left(x\right)^{\alpha}\frac{1}{\left\vert B_{R}\right\vert }\int_{B_{R}}w\left(y\right)^{-\alpha}dy\text{ if }\alpha<0.
\]
Hence, for $R\in\left(0,1\right]$,
\begin{eqnarray}
 &  & \left(\frac{1}{\left\vert B_{R}\right\vert }\int_{B_{R}}w\left(x+y\right)^{\alpha}dy\right)\left(\frac{1}{\left\vert B_{R}\right\vert }\int_{B_{R}}w\left(x+y\right)^{-\alpha q/p}dy\right)^{p/q}\nonumber \\
 & \leq & C\left(\frac{1}{\left\vert B_{R}\right\vert }\int_{B_{R}}w\left(y\right)^{\left\vert \alpha\right\vert }dy\right)\left(\frac{1}{\left\vert B_{R}\right\vert }\int_{B_{R}}w\left(y\right)^{\left|\alpha\right|q/p}dy\right)^{p/q}\leq C.\label{eq:onlyy_2}
\end{eqnarray}
Let $R>1,\left\vert x\right\vert >2R$. Then with $\left\vert y\right\vert \leq R$
we have $2\left\vert x\right\vert \geq\left\vert y+x\right\vert \geq\frac{1}{2}\left\vert x\right\vert $.
Hence for each $\alpha\in\mathbf{R},$ there is $C>0$ so that
\[
\frac{1}{\left\vert B_{R}\right\vert }\int_{B_{R}}w\left(x+y\right)^{\alpha}dy\leq Cw\left(x\right)^{\alpha}.
\]

\noindent The conclusion follows from (\ref{eq:onlyy_1}) and (\ref{eq:onlyy_2}).

\noindent Let $R>1,\left\vert x\right\vert \leq2R$. By (\ref{fa2}),
\begin{eqnarray*}
\frac{1}{\left\vert B_{R}\right\vert }\int_{B_{R}}w\left(x+y\right)^{\alpha}dy & \leq & C\frac{1}{\left\vert B_{3R}\right\vert }\int_{\left\vert x+y\right\vert \leq3R}w\left(x+y\right)^{\alpha}dy\\
 & \leq & CR^{\alpha}.
\end{eqnarray*}
Hence,
\[
\left(\frac{1}{\left\vert B_{R}\right\vert }\int_{B_{R}}w\left(x+y\right)^{\alpha}dy\right)\left(\frac{1}{\left\vert B_{R}\right\vert }\int_{B_{R}}w\left(x+y\right)^{-\alpha q/p}dy\right)^{\frac{p}{q}}\leq C,\,x\in\mathbf{R}^{d},\,R>0.
\]
\end{proof}

\subsection{Newton Potential Estimates}

We investigate some fundamental properties of Newton potential in
weighted Sobolev spaces. 
\begin{lem}
\label{Lem:lle1} Let $p>1,\,$$\Delta u=0$ and $u\in H_{\delta,p}^{0}\left(\mathbf{R}^{d}\right)$
with $\delta\geq0$. Then $u=0$.
\end{lem}
\begin{proof}
Let $u_{1},u_{2}\in H_{\delta,p}^{0}\left(\mathbf{R}^{d}\right)$
be the solutions of $\Delta u=0.$ Let $v=u_{1}-u_{2}$ and 
\[
v_{\varepsilon}=v\ast\varphi_{\varepsilon}
\]
with 
\[
\varphi_{\varepsilon}\left(x\right)=\varepsilon^{-d}\varphi\left(x/\varepsilon\right),\hspace{1em}x\in\mathbf{R}^{d}\text{,}
\]
and $\varphi\in C_{0}^{\infty}\left(\mathbf{R}^{d}\right),\varphi\geq0,\int\varphi=1$.

\noindent As a bounded harmonic function $v_{\epsilon}=c.$ By definition,
$cw^{\delta-d/p}\in L_{p}\left(\mathbf{R}^{d}\right)$ and thus $c=0.$
\end{proof}
For $f\in C_{0}^{\infty}\left(\mathbf{R}^{d}\right)$ we denote $N\left(f\right)$
the Newton potential of $f:$
\[
N\left(f\right)=\int\Gamma\left(\cdot-y\right)f\left(y\right)dy.
\]

\begin{lem}
\label{Lem:lle2}Let $p>1,d\geq2,\theta\in\left(1,d\right)$. Then
there exists $C>0$ such that for all $f\in H_{\theta,p}^{0}\left(\mathbf{R}^{d}\right)$
and $u=\nabla N\left(f\right)$, i.e., $u\left(x\right)=\int\nabla\Gamma\left(x-y\right)f\left(y\right)dy,\,x\in\mathbf{R}^{d},$

\noindent
\[
\left\vert u\right\vert _{H_{\theta,p}^{1}\left(\mathbf{R}^{d}\right)}\leq C\left\vert f\right\vert _{H_{\theta,p}^{0}}.
\]
\end{lem}
\begin{proof}
Let $f\in C_{0}^{\infty}\left(\mathbf{R}^{d}\right),\psi=N\left(f\right)$
and 

\noindent
\[
u\left(x\right)=\nabla\psi\left(x\right)=\int\nabla\Gamma\left(x-y\right)f\left(y\right)dy,\hspace{1em}x\in\mathbf{R}^{d}.
\]

\noindent By the estimate of \citep[Theorem 9.9]{gt}, for each $p>1$
there exists $C>0$ such that
\[
\left\vert \nabla u\right\vert _{p}\leq C\left\vert f\right\vert _{p}.
\]

\noindent In particular, $\left|\nabla u\right|_{2}\leq C\left|f\right|_{2}$.
Also, it is straightforward to verify that $\left|D^{\alpha}D^{2}\Gamma\left(x\right)\right|\leq\left|x\right|^{-d-\alpha}$
for all $\left|x\right|\neq0$ and $\left|\alpha\right|\leq1$. Hence,
by Lemma \ref{al1}, \citep[Theorem 2, Chapter V.4]{st} and \citep[Theorem 1, Chapter V.3]{st},
for each $\theta\in\left(1,d\right)$ and $p>1$ there is $C>0$ so
that

\noindent
\[
\left\vert w^{\theta-d/p}\nabla u\right\vert _{p}\leq C\left\vert w^{\theta-d/p}f\right\vert _{p}.
\]

\noindent We will now apply generalized Hardy-Littlewood inequality
(see e.g. \citep[Theorem 1.3.5]{s}) to show that
\begin{equation}
\left\vert w^{\theta-d/p-1}u\right\vert _{p}\leq C\left\vert w^{\theta-d/p}f\right\vert _{p}.\label{eq:u}
\end{equation}

\noindent Consider, 
\begin{align*}
\left|w^{\theta-d/p-1}\left(x\right)u\left(x\right)\right| & =C\left|\int\frac{w\left(x\right)^{\theta-d/p-1}}{w\left(y\right)^{\theta-d/p}}\frac{x_{i}-y_{i}}{\left|x-y\right|^{d}}w\left(y\right)^{\theta-d/p}f\left(y\right)dy\right|\\
 & \leq C\int\frac{w\left(x\right)^{\theta-d/p-1}}{w\left(y\right)^{\theta-d/p}}\left|x-y\right|^{1-d}w\left(y\right)^{\theta-d/p}\left|f\left(y\right)\right|dy.
\end{align*}

\noindent Case i. $\theta-\frac{d}{p}-1\geq0$,

\[
\frac{w\left(x\right)^{\theta-d/p-1}}{w\left(y\right)^{\theta-d/p}}\leq C\left(\left|x\right|^{\theta-d/p-1}\left|y\right|^{-\left(\theta-d/p\right)}+\left|y\right|^{-1}\right).
\]
Case ii. $\theta-\frac{d}{p}-1<0$,
\[
\frac{w\left(x\right)^{\theta-d/p-1}}{w\left(y\right)^{\theta-d/p}}\leq C\left(\left|x\right|^{\theta-d/p-1}\left|y\right|^{-\left(\theta-d/p\right)}\right).
\]

\noindent Because $\theta\in\left(1,d\right)$, the condition of \citep[Theorem 1.3.5]{s}
is now easily verified and (\ref{eq:u}) is proved. Therefore,

\[
\left\vert u\right\vert _{H_{\theta,p}^{1}}\leq C\left\vert f\right\vert _{H_{\theta,p}^{0}}.
\]

\noindent The estimate for general $f\in H_{\theta,p}^{0}\left(\mathbf{R}^{d}\right)$
follows by passing to the limit. 
\end{proof}
We will need higher order estimates of the Newton potential. 
\begin{lem}
\label{Lem:lle3}Let $p>1,\,l\geq0,\,d\geq2,\,\theta\in\left(1,d\right)$.
Then there exists $C>0$ so that for all $f\in H_{\theta+l,p}^{l}\left(\mathbf{R}^{d}\right),$
\begin{equation}
\left\vert \nabla N\left(f\right)\right\vert _{H_{\theta+l,p}^{l+1}}\leq C\left\vert f\right\vert _{H_{\theta+l,p}^{l}}.\label{fo2}
\end{equation}
\end{lem}
\begin{proof}
We prove the claim by induction. The case of $l=0$ follows from Lemma
\ref{Lem:lle2}. Let $\psi=N\left(f\right),f\in C_{0}^{\infty}\left(\mathbf{R}^{d}\right)$.
Assume that we proved (\ref{fo2}) for $0\leq l\leq n$ and for each
$p>1$. Consider for a multi-index $\beta\in\mathbf{N}_{0}^{d}$ with
$\left\vert \beta\right\vert =n+1,$ 
\[
\psi_{\beta}=D^{\beta}\psi=N\left(f_{\beta}\right),
\]
where $f_{\beta}=D^{\beta}f.$ Then
\[
\Delta\psi_{\beta}=f_{\beta},\hspace{1em}\Delta\left(\nabla\psi_{\beta}\right)=\nabla f_{\beta}\text{ in }\mathbf{R}^{d},
\]
and $v=w^{n+1}\psi_{\beta}$ solves
\[
\Delta v=F,\hspace{1em}\Delta\left(\nabla v\right)=\nabla F\text{,}
\]
where $F=w^{n+1}f_{\beta}+\Delta\left(w^{n+1}\right)\psi_{\beta}+2\nabla\left(w^{n+1}\right)\cdot\nabla\psi_{\beta}.$

\noindent By (\ref{fo2}) for $l=n$, 
\begin{align}
\left\vert w^{\theta-d/p+n-1}\psi_{\beta}\right\vert _{p}+\left\vert w^{\theta-d/p+n}\nabla\psi_{\beta}\right\vert _{p} & \leq\left\vert f\right\vert _{H_{\theta+n,p}^{n}}.\label{eq:phi_beta}
\end{align}

\noindent Therefore,
\begin{eqnarray}
\left\vert w^{\theta-d/p}F\right\vert _{p} & \leq & C\left[\left\vert w^{\theta-d/p+n+1}f_{\beta}\right\vert _{p}+\left\vert w^{\theta-d/p+n-1}\psi_{\beta}\right\vert _{p}+\left\vert w^{\theta-d/p+n}\nabla\psi_{\beta}\right\vert _{p}\right]\label{eq:F_H0}\\
 & \leq & C\left\vert f\right\vert _{H_{\theta+n+1,p}^{n+1}}.\nonumber 
\end{eqnarray}
Now, 
\begin{eqnarray*}
w^{\theta-d/p-1}\nabla v & = & w^{\theta-d/p-1}\left[\nabla\left(w^{n+1}\right)\psi_{\beta}+w^{n+1}\nabla\psi_{\beta}\right]\\
 & = & \left(n+1\right)w^{\theta-d/p+n-1}\nabla w\psi_{\beta}+w^{\theta-d/p+n}\nabla\psi_{\beta},
\end{eqnarray*}
and by (\ref{eq:phi_beta}),
\[
\left\vert w^{\theta-d/p-1}\nabla v\right\vert _{p}\leq C\left\vert f\right\vert _{H_{\theta+n,p}^{n}}.
\]
Clearly, $\nabla N\left(F\right)$ solves $\Delta\left(\nabla N\left(F\right)\right)=\nabla F$.
Due to Lemma \ref{Lem:lle2} and (\ref{eq:F_H0}), 
\begin{equation}
\left\vert \nabla N\left(F\right)\right\vert _{H_{\theta,p}^{1}}\leq C\left\vert F\right\vert _{H_{\theta,p}^{0}}\leq C\left\vert f\right\vert _{H_{\theta+n+1,p}^{n+1}}<\infty.\label{eq:LpNv}
\end{equation}

\noindent Hence,
\[
\left\vert w^{\theta-d/p-1}\left(\nabla v-\nabla N\left(F\right)\right)\right\vert _{p}\leq C\left\vert f\right\vert _{H_{\theta+n+1,p}^{n+1}}<\infty.
\]
We conclude by Lemma \ref{Lem:lle1} that $\nabla v=\nabla N\left(F\right).$

\noindent From (\ref{eq:LpNv}) for each $p>1,$ there is $C>0$ so
that for all multi-index $\mu\in\mathbf{N}_{0}^{d}$ with $\left|\mu\right|=1$,
\[
\left\vert w^{\theta-d/p}D^{\mu}\nabla v\right\vert _{p}\leq C\left\vert f\right\vert _{H_{\theta+n+1,p}^{n+1}}.
\]
We have, recalling that $v=w^{n+1}\psi_{\beta}$, $\nabla v=\left(n+1\right)w^{n}\nabla w\psi_{\beta}+w^{n+1}\nabla\psi_{\beta}$
and 
\begin{eqnarray*}
D^{\mu}\nabla v & = & \left(n+1\right)\left[nw^{n-1}D^{\mu}w\nabla w+w^{n}D^{\mu}\nabla w\right]\psi_{\beta}\\
 & + & \left(n+1\right)w^{n}D^{\mu}w\nabla\psi_{\beta}+w^{n+1}D^{\mu}\nabla\psi_{\beta}\\
 & = & B+w^{n+1}D^{\mu}\nabla\psi_{\beta},
\end{eqnarray*}
where by (\ref{eq:phi_beta})
\begin{eqnarray*}
\left\vert w^{\theta-d/p}B\right\vert _{p} & \leq & C\left(\left\vert w^{\theta-d/p+n-1}\psi_{\beta}\right\vert _{p}+\left\vert w^{\theta-d/p+n}\nabla\psi_{\beta}\right\vert _{p}\right)\\
 & \leq & C\left\vert f\right\vert _{H_{\theta+n,p}^{n}}.
\end{eqnarray*}
Hence, 
\[
\left\vert w^{\theta-d/p+n+1}D^{\mu}\nabla\psi_{\beta}\right\vert _{p}\leq C\left\vert f\right\vert _{H_{\theta+n+1,p}^{n+1}}
\]
and the statement is proved. The estimate for general $f\in H_{\theta+l,p}^{l}\left(\mathbf{R}^{d}\right)$
follows by passing to the limit.
\end{proof}

\section*{Conclusion}

We employed the Euler-Lagrangian approach to prove the local existence
of the Navier-Stokes equations in weighted Sobolev spaces on the full
domain. This paper is the first to cover general $p>d\geq2$. In the
future, we will expand our approach and prove similar results with
stochastic integrals as the forcing terms. 

\section*{Declarations}

\subsection*{Ethical approval }

not applicable

\subsection*{Funding}

This research was funded by College of Industrial Technology, King
Mongkut's University of Technology North Bangkok (Grant No. Res-CIT320/2023).

\subsection*{Availability of data and materials }

not applicable

\subsection*{Conflict of interest}

On behalf of all authors, the corresponding author states that there
is no conflict of interest.

\end{document}